\documentclass{amsart}

\usepackage{amssymb}
\usepackage{graphicx,color}
\usepackage{amsmath}
\usepackage{url}
\usepackage[colorlinks=true]{hyperref}

\usepackage{cleveref}
\Crefname{counterexample}{counterexample}{counterexamples}
\Crefname{counterexample}{Counterexample}{Counterexamples}
\Crefname{conjecture}{conjecture}{conjectures}
\Crefname{conjecture}{Conjecture}{Conjectures}

\usepackage{tikz-cd}
\usepackage{tikz}
\usetikzlibrary{matrix}

\theoremstyle{plain}
\newtheorem{theorem}{Theorem}[section]
\newtheorem{proposition}[theorem]{Proposition}
\newtheorem{lemma}[theorem]{Lemma}
\newtheorem{corollary}[theorem]{Corollary}

\theoremstyle{definition}
\newtheorem{definition}[theorem]{Definition}

\newtheorem{example}[theorem]{Example}

\newtheorem{question}[theorem]{Question}
\newtheorem{conjecture}[theorem]{Conjecture}
\newtheorem{remark}[theorem]{Remark}
\newtheorem{algorithm}[theorem]{Algorithm}

\newcommand{\N}{ \ensuremath{\mathbb{N}}}
\newcommand{\Z}{ \ensuremath{\mathbb{Z}}}

\newcommand{\R}{ \ensuremath{\mathbb{R}}}

\newcommand{\ZZ}{ \ensuremath{\mathbb{Z}}}

\newcommand{\Fibo}{F}

\renewcommand{\int}{\operatorname{int}}
\newcommand{\width}{\operatorname{width}}
\newcommand{\conv}{\operatorname{conv}}
\newcommand{\aff}{\operatorname{aff}}
\newcommand{\cone}{\operatorname{cone}}
\newcommand{\Cyc}{\operatorname{Cycl}}
\newcommand{\Vol}{\operatorname{Vol}}

\DeclareMathOperator{\arcsinh}{arcsinh}

\newcommand{\vi}[1][i]{v^{(#1)}}
\newcommand{\uj}[1][j]{u^{(#1)}}
\newcommand{\A}[2]{S(#1,#2)}
\newcommand{\up}[1]{{}^{(#1)}}
\newcommand{\flt}{\operatorname{Flt}}

\newcommand{\blue}[1]{{\color{black} #1}}

\usepackage[colorinlistoftodos]{todonotes}

\newcommand{\ra}{\Rightarrow}
\newcommand{\lra}{\Leftrightarrow}

\date{\today}

\author[J.~Doolittle]{Joseph Doolittle}
\author[L.~Katth\"an]{Lukas Katth\"an}
\author[B.~Nill]{Benjamin Nill}
\author[F.~Santos]{Francisco Santos}

\address[J.~Doolittle]
{
Institut f\"ur Geometrie, TU Graz, Austria
}
\email{jdoolittle@tugraz.at}

\address[L. ~Katth\"an]
{
Institut f\"ur Mathematik,
Goethe-Universit\"at Frankfurt,
Germany
}
\email{lukaskatthaen@gmx.de}

\address[B. ~Nill]
{
Otto-von-Guericke Universit\"at-Magdeburg, Germany.}
\email{benjamin.nill@ovgu.de}

\address[F.~Santos]
{
Dep. of Mathematics, Statistics and Comp. Sci., 
Univ. of Cantabria, Spain
}
\email{francisco.santos@unican.es}

\title{Empty simplices of large width}

\thanks{Work of F. Santos is supported by grants  PID2019-106188GB-I00 and PID2022-137283NB-C21  funded by MCIN/AEI/10.13039/501100011033, by the Einstein Foundation Berlin under grant EVF-2015-230 and by
project CLaPPo (21.SI03.64658) of Universidad de Cantabria and Banco Santander. B. Nill has been a PI in the Research Training Group Mathematical Complexity Reduction funded
by the Deutsche Forschungsgemeinschaft (DFG, German Research Foundation) - 314838170, GRK 2297
MathCoRe. }

\begin{document}

\begin{abstract}An empty simplex is a lattice simplex in which vertices are the only lattice points.
We show two constructions leading to the first known empty simplices of width larger than their dimension:

\begin{itemize}

\item We introduce \emph{cyclotomic simplices} and exhaustively compute all the cyclotomic simplices of dimension $10$ and volume up to $2^{31}$.
Among them we find five empty ones of width $11$, and none of larger width.

\item Using \emph{circulant} matrices of a very specific form, we construct empty simplices of arbitrary dimension $d$ and width growing asymptotically as  $d/\arcsinh(1) \sim 1.1346\,d$.
\end{itemize}


\end{abstract}

\maketitle

\section{Introduction and summary of results}

Let $\Lambda\subset \R^d$ be a lattice. The \emph{flatness theorem} \cite{Banaszczyk_etal1999, Khinchine, KannanLovasz1988} states that the maximum lattice width among all hollow convex bodies in $\R^d$ is bounded by a constant $\flt(d)$ depending solely on $d$.
Here, a \emph{convex body} is a compact convex subset of $\R^d$ with non-empty interior, and a convex body
$K$ is \emph{hollow} or \emph{lattice-free} (with respect to the lattice $\Lambda$) if the interior of $K$ does not contain any lattice point. \blue{The \emph{width of $K$ with respect to a functional} $f\in (\R^d)^*$ is the length of the interval $f(K)$} and
the \emph{lattice width} (or just \emph{width}, for short) of $K$ is the minimum width it has with respect to $f\in \Lambda^*\setminus \{0\}$. \blue{Note that the width depends on the lattice $\Lambda$ being considered. If $K$ is a lattice polytope then its width is a nonnegative integer.}

The best current upper bound for the width of general hollow convex bodies is $O(d^{4/3}\log^a d)$ for some constant $a$~\cite{Rudelson2000}. But no construction of hollow convex bodies of width larger than linear is known, 
so the general belief is that $\flt(d) \in O(d)$, perhaps modulo polylogarithmic factors. 
In fact, this is known to hold for some classes of convex bodies such as simplices and centrally symmetric bodies, for which width is known to be bounded by $O(d\log d)$~\cite{Banaszczyk1996,Banaszczyk_etal1999}.

\emph{Empty simplices}, that is, lattice simplices with no lattice points apart from their vertices, play an important role in the theory of lattice polytopes for two reasons. On the one hand, every lattice polytope can be triangulated into empty simplices, so results about empty simplices often have implications for general lattice polytopes. On the other hand, empty simplices (translated to have a vertex at the origin) correspond to \emph{terminal abelian finite quotient singularities}, of interest in the minimal model program (see, e.g., \cite[Chapter~11]{CoxLittleSchenck2011} or \cite[Chapter~14]{Matsuki2002}). 

\blue{
The starting motivation for this paper is that,
experimentally, empty simplices of large width have an affine symmetry acting cyclically over their vertices (see details below). This suggests it is interesting to study systematically empty simplices with this property. We do so in Section~\ref{sec:cyclic} and then study two constructions of them, that we call \emph{cyclotomic} and \emph{circulant} simplices, in Sections~\ref{sec:cyclotomic} and \ref{sec:circulant}.

Cyclotomic simplices include all the simplices of prime volume  with such a cyclic symmetry (Corollary~\ref{coro:cyclotomic2}). There is one for each prime volume $N$ and dimension $d$ such that $d+1$ divides $N-1$. We denote it $\Cyc(d,N)$. Circulant simplices are defined via cyclic permutations of vertex coordinates. We look at a particular family of them, denoted
 $\A{d}{m}$ where $d$ is the dimension and $m\in \N$ is a parameter used in their definition.
See more details in the discussion below, and precise definitions in Sections \ref{sec:cyclotomic} and \ref{sec:circulant}.
}

\subsection*{Main results} 
In this paper \blue{we use these two constructions to produce} the first examples of empty simplices of width larger than their dimension. The following table summarizes our findings in small dimensions. In it, and elsewhere, volume is normalized to the lattice; unimodular simplices have volume one.

\begin{theorem}
\label{thm:small}
There are empty simplices of the following dimensions and widths:
\rm
\small
\[
\begin{array}{c|c|cc}
\text{empty simplex}& \text{volume} & \text{dimension} & \text{width}\\
\hline
\Cyc(4, \ 101) = \A{4}{2} &101 & 4 & 4\\
\Cyc(6, \ 6\,301) = \A{6}{3} &6\,301 & 6 & 6\\
\A{8}{4} & 719\,761 & 8 & 8\\
\Cyc(10, \ 582\,595\,883)  &582\,595\,883 & 10 & {\bf 11}\\
\A{16}{9} &36\,373\,816\,216\,801\,891 & 16 & {\bf 18}\\
\A{30}{17} &  \simeq 2.9{\cdot} 10^{38}& 30 & {\bf 34}\\
\A{46}{26} &  \simeq 6.6{\cdot} 10^{66}& 46 & {\bf 52}\\
\A{60}{34} &  \simeq 5.4{\cdot} 10^{93}& 60 & {\bf 68}\\
\end{array}
\]
\end{theorem}

The first entry in the table was known; it is the empty $4$-simplex of width four found by Haase and Ziegler~\cite{HaaseZiegler2000}, and shown by Iglesias and Santos~\cite{IglesiasSantos2019} to be the unique empty $4$-simplex of width larger than three. 
We include it in the table to emphasize that our constructions reflect two ways of understanding \blue{it, and of }generalizing the unimodular symmetry that acts cyclically on its vertices.

But even our other simplices of width \emph{equal} to their dimension are new. 
It is very easy to construct hollow lattice $d$-simplices of width $d$ (the $d$-th dilation of a unimodular $d$-simplex is one example), but for empty simplices the same is not true; e.g., all empty 3-simplices have width one.
In fact, the first proof that there are empty simplices of arbitrarily large width (in growing dimension) was non-constructive, by Kantor~\cite{Kantor}, and the widest empty simplices before this paper in every dimension $d>4$ are those constructed by Seb\H{o}~\cite{Sebo1999}, of width $d-1$ for even dimension $d$, and width $d-2$ when $d$ is odd. 

\blue{Let us } denote by $\flt_e(d)$ the maximum lattice width among empty $d$-simplices, \blue{which must be attained since it is a positive integer bounded by $\flt(d)$}.
What we know  about $\flt_e(d)$ can be summarized as follows:

\begin{theorem}[partially Codenotti and Santos~\cite{CodenottiSantos2020}]
\label{thm:limit}
\hspace{0pt}
\begin{enumerate}
\item The sequence $(\flt_e(d))_{d\in \N}$ is weakly increasing: $\flt_e(d+1) \ge \flt_e(d)$ for all $d$.
\item $\flt_e(d(m+3)) \ge m\, \flt_e(d), \quad \forall m,d\in \N$.
\item
$\lim_{d\to \infty} \frac{\flt_e(d)}{d}$ exists (in $\R \cup \{\infty\}$) and it equals  $\sup_{d\in \N} \frac{\flt_e(d)}d$.
\end{enumerate}
\end{theorem}

\begin{proof}
For part (1), let $\Delta=\conv(0, v_1,\dots,v_d)$ be an empty simplex of width $\flt_e(d)$, which we assume to have a vertex at the origin without loss of generality. Let $f\in (\R^d)^*$ be the linear functional with $f(v_i)=1$ for all $i$ and let
$v_{d+1}$ be a lattice point minimizing $f$ among the lattice points in the interior \blue{of the cone $\cone(v_1,\dots,v_d)$ generated by the vertices of $\Delta$}. 

\blue{Since $\conv(0, v_1,\dots,v_d, v_{d+1})$ is contained in 
\[
\cone(v_1,\dots,v_d) \cap f^{-1}([0,f(v_{d+1})]),
\]}%
minimality of $v_{d+1}$ implies that the polytope $\conv(0, v_1,\dots,v_d, v_{d+1})$ has no lattice points other than its vertices. In particular, for any choice of $h\in \Z\setminus\{0\}$ we have that 
\[
\tilde \Delta(h) :=\conv( (0,0), (v_1,0),\dots,(v_{d}, 0), (v_{d+1},h))
\]
is an empty $(d+1)$-simplex. If we take $h\ge \flt_e(d)$ we have that $\tilde \Delta(h)$ has width \blue{at least } $\flt_e(d)$  because:
\begin{itemize}
\item The width of $\tilde \Delta(h)$ with respect to any lattice functional not constant on $\R^d\times \{0\}$ is at least the width of 
$\tilde \Delta(h) \cap (\R^d\times \{0\})$, which equals $\flt_e(d)$.
\item The width of $\tilde \Delta(h)$ with respect to any lattice functional $g$ constant on $\R^d\times \{0\}$ equals \blue{$g(h)$, which is a positive multiple of $h$}.
\end{itemize}

Part (2) is proved as Theorem 1.8 in \cite{CodenottiSantos2020}, with a construction related to Seb\H{o}'s. Part (3) is also in \cite[Theorem 1.8]{CodenottiSantos2020}, except that there a $\limsup$ appears instead of $\lim$. But the monotonicity in part (1) allows an easy modification of that proof to give a $\lim$.
\end{proof}

Our constructions imply that the $\lim_{d\to \infty} \frac{\flt_e(d)}{d} $ in the previous theorem is at least 1.1346 (see Theorem~\ref{thm:main}), in particular disproving the following guess from \cite[p.~403]{Sebo1999}: 
``it seems to be reasonable to think that the maximum width of an empty integer simplex in $\R^n$ is $n$ + constant''. 

\subsection*{Cyclotomic and circulant simplices}

As suggested by our notation for the simplices in Theorem~\ref{thm:small}, we have constructions of two different types:

In Section~\ref{sec:cyclotomic} we introduce \emph{cyclotomic simplices}: lattice simplices of prime volume $N$ constructed using primitive roots of unity modulo $N$.
There is one cyclotomic simplex $\Cyc(d,N)$ for each choice of dimension $d$ and prime volume $N$, as long as $d+1$ divides $N-1$ (so that $(d+1)$-th roots of unity modulo $N$ exist).
In order for them to have large width, $d+1$ has to be prime too (see Corollary \ref{coro:cyclotomic}).
 We have computationally enumerated \emph{empty} cyclotomic simplices of dimensions 4, 6 and 10 and of volume up to $2^{31}$  and found that \blue{in dimension 10 there are five of width 11. See a more detailed discussion in Section~\ref{sec:cyclotomic}}.
%

In Section~\ref{sec:circulant} we present another construction, using simplices with explicit coordinates that are symmetric under cyclic permutation. 
We call them \emph{circulant simplices} since their matrix of vertex coordinates is a circulant matrix. 
Our exploration of circulant simplices starts with the observation that some of the widest simplices known, including the cyclotomic simplices $\Cyc(4, 101)$ and $\Cyc(6, 6\,301)$, are circulant of a very specific form. 
Let $v=(1,m,0,\cdots,0, -m)\in \R^{d+1}$, where $m$ is a positive integer and \(v\) has $d-2$ zero entries. Let $\vi$ be the $i$-th cyclic shift of $v$. We define
\[
\A{d}{m}:= \conv(\vi[0], \dots,\vi[d]),
\]
which is a $d$-simplex in the hyperplane $\sum_i x_i=1$ in $\R^{d+1}$. 
For odd $d$ this simplex has width one, but for even $d$ 
\blue{the situation is quite different. The following statement summarizes the main results of Section~\ref{sec:circulant}. The $m_0$ in the statement is
%
defined as the unique positive solution of the polynomial equation
\[
m^{d-1} = \sum_{i=0}^{d/2-1} \binom{d-1-i}{i}m^{2i}.
\]
Equivalently, it equals 
$
\frac{1}{2\sinh\alpha_0},
 $
where $\alpha_0=\alpha_0(d)$ is the unique solution of $\cosh \alpha_0 = \sinh (d\alpha_0)$. See details in Section~\ref{sec:asymptotics}.

\begin{theorem}
In each even dimension $d\ge 4$ and for every $m\le \lfloor m_0(d)\rfloor$, the circulant $d$-simplex $S(d,  m)$
is empty and has width $2 m$. Hence:
\[
\lim_{d\to \infty} \frac{\flt_e(d)}{d} \ge  \lim_{d\to \infty} \frac{2m_0(d)}{d} = \frac{1}{\arcsinh(1)} = \frac{1}{\ln(1+\sqrt2)} \simeq 1.1346.
\]
\label{thm:main}
\end{theorem}
}


The dimensions $16$, $30$, $46$ and $60$ that appear in Theorem~\ref{thm:small}
are the smallest values of $d$ for which $2m_0(d)$ exceeds, respectively, $d+2$, $d+4$, $d+6$, and $d+8$. 

\goodbreak
\subsection*{Discussion and questions} 

\subsubsection*{Asymptotics}
The main open question on this topic is 

\begin{question}
Is $\lim_{d\to \infty} \frac{\flt_e(d)}{d}$ equal, or close to, $\frac1{\arcsinh(1)}$? Is it finite?
\end{question}

Another important question is whether this limit is equal, similar, or significantly different from the same limit for the  original flatness constant \blue{$\flt(d)$ in which arbitrary convex bodies are allowed.} The best lower bound that we know for the latter is $2$, which follows from a very recent construction of  lattice-free simplices of width close to $2d$~\cite{MSW21}. (Note, however, that their simplices are not lattice simplices).

\subsubsection*{Parity}
Empty simplices of large width seem easier to construct in even dimensions. Apart from the fact that in both of our constructions we need $d$ to be even, the construction of Seb\H{o}~\cite{Sebo1999} gives a better result in even dimensions. Moreover, we know that $\flt_e(2)=\flt_e(3)=1$, while $\flt_e(4)=4$. This suggests the following possibility:

\begin{question}
Is $\flt_e(2k+1) = \flt_e(2k)$ for every $k\in \N$?
\end{question}

Observe that for the original flatness constant the same does not happen, since $\flt(d+1) \ge \flt(d)+1$ for every $d$ (\cite[Proposition 1.7]{CodenottiSantos2020}).

\subsubsection*{Symmetric empty simplices}
Both of our constructions produce empty simplices that have a unimodular
symmetry acting cyclically on vertices. In fact, our work can be understood as an exploration of such symmetries, and is \blue{inspired} by the same type of symmetry in other constructions of convex bodies of large width:

\begin{itemize}
\item The empty simplices of even dimension $d$ and width $d-1$ constructed by Seb\H{o}~\cite{Sebo1999} are in fact circulant simplices.

\item Hurkens~\cite{Hurkens1990} established that $\flt(2) = 1 + 2/\sqrt3\sim 2.14$ by showing that the unique (modulo unimodular equivalence) convex $2$-body of maximum lattice width is a certain triangle.  Codenotti and Santos~\cite{CodenottiSantos2020} conjecture that $\flt(3) = 2 + \sqrt2\sim 3.41$, attained by a certain tetrahedron. 
(Averkov et al.~\cite{ACMS} have shown that this tetrahedron is at least a local maximizer of width among all hollow convex $3$-bodies).
Both the Hurkens triangle and the Codenotti-Santos tetrahedron have a unimodular symmetry permuting vertices cyclically.
\end{itemize}

The fact that $\flt(2)$ and $\flt_e(4)$ are attained \emph{only} at symmetric examples is an indication that the role of symmetry in the constructions is not merely a tool to construct relatively complicated objects with simple rules.

\begin{question}
Is it true that in every dimension there are symmetric convex bodies/simplices of width attaining (or close to attaining) $\flt(d)$ and/or $\flt_e(d)$?
In the other direction, can one improve on the known upper bounds for $\flt(d)$ or $\flt_e(d)$ if they are restricted to symmetric instances?
\end{question}

For the second part of this question it might be useful to clarify the exact relation between cyclotomic simplices, circulant simplices, and simplices that admit a symmetry shifting vertices cyclically.  Lemma~\ref{lemma:cyclotomic} is a partial answer, but we do not know the answer to the following:

\begin{question}
\label{q:circulant}
Does every empty lattice simplex with a unimodular symmetry acting cyclically on vertices admit a circulant representation with integer coefficients?
\end{question}

\subsubsection*{Finiteness}
Our lists of cyclotomic simplices in dimensions $4$ and $6$ are most probably complete. Our list in dimension 10 might not be complete, but it should be very close. At any rate, we conjecture that there are finitely many empty cyclotomic $d$- simplices for each prime $d+1$. More generally:

\begin{conjecture}
\label{conj:finite}
If $d+1$ is prime, then there are only finitely many (modulo unimodular equivalence) empty $d$-simplices with a unimodular symmetry acting transitively on vertices.
\end{conjecture}

This question and, more generally, finiteness of empty \emph{orbit polytopes} for an arbitrary group acting as permutation of coordinates in $\R^n$, is studied in \cite{HerrRehnSchuermann15, LadishSchuermann}. For example, Theorem 5.1 in \cite{LadishSchuermann} (see also their Example 5.9) implies Conjecture~\ref{conj:finite} for the case of circulant $d$-simplices. (Of course, a positive answer to Question~\ref{q:circulant} would imply the general case from the circulant case).
In fact, our conjecture is a special case of the conjecture stated in the second paragraph of \cite[Section 5]{LadishSchuermann}. 

Necessity of $d+1$ being prime in Conjecture~\ref{conj:finite} comes from the fact that if $d+1$ is composite then the cyclic action factors as an action permuting blocks of coordinates and an action within blocks. By \cite[Theorem 28]{HerrRehnSchuermann15} this implies the existence of infinitely many empty orbit polytopes. 
We show an explicit infinite family for every composite $d+1$ in Example~\ref{exm:nonprime_circulant}. Similar examples are described in \cite[Example 5.24]{RehnThesis} or \cite[Example 2.7]{LadishSchuermann}.

\subsubsection*{Quotient group, cyclicity}
The geometry of a lattice simplex $\Delta$ is closely related to the quotient group $G_\Delta:=\Lambda/\Lambda_\Delta$ of the ambient lattice $\Lambda$ to the lattice $\Lambda_\Delta$ generated by the vertices of $\Delta$. 
This quotient group is a finite abelian group of order equal to the (normalized) volume of $\Delta$. Simplices for which this quotient group is cyclic are called \emph{cyclic simplices} (see Section~\ref{sec:cyclic}).

$G_\Delta$ is, in turn, related to the volume of facets. In particular, part (2) of Proposition~\ref{prop:notcyclic} implies that simplices with a unimodular facet
are cyclic  (but not vice-versa).

\begin{question} 
If \(d\) is even, is $\flt_e(d)$ always attained by a cyclic simplex?
And by a  simplex with all facets unimodular?
\end{question}

Motivation for this question is that cyclotomic simplices have unimodular boundary, and the circulant simplices $S(d,m)$ considered in this paper either have unimodular boundary or have all facets of volume two (in which case their quotient group is isomorphic to $\Z_{N/2} \oplus \Z_2$; see Theorem~\ref{thm:boundary}). In particular, all the simplices in Theorem~\ref{thm:small} are cyclic.
Furthermore, Seb\H{o}'s  even-dimensional empty simplices of width $d-1$ have all facets unimodular.

\subsection*{Organization of the paper} Section~\ref{sec:cyclic} recalls some results about cyclic lattice simplices. Section~\ref{sec:cyclotomic} presents our investigations on empty cyclotomic lattice simplices. Finally, Section~\ref{sec:circulant} contains the construction of empty circulant lattice simplices.

\subsection*{Acknowledgments} 
We thank the Mathematisches Forschungsinstitut Oberwolfach and the organizers of the workshop
\emph{Geometric, Algebraic, and Topological Combinatorics} (August 2019)
where this research was started. 
We also want to thank Johannes Hofscheier for helpful discussions, and an anonymous referee for a very attentive reading and helpful comments and corrections.

\section{Preliminaries on cyclic lattice simplices}
\label{sec:cyclic}

Let $\Lambda\cong \Z^d$ be a lattice in $\R^d$. Recall that a \emph{lattice simplex} is a simplex with vertices in $\Lambda$, that is, $\Delta=\conv(v_0,\dots,v_d)$ where $v_0,\dots,v_d \in \Lambda$ are affinely independent. Two lattice simplices $\Delta$ and $\Delta'$ for respective lattices $\Lambda, \Lambda'$ are \emph{unimodularly equivalent} if there is an affine lattice isomorphism $f:\Lambda \to \Lambda'$ (that is, an affine map sending $\Lambda$ bijectively to $\Lambda'$) such that $f(\Delta) = \Delta'$. We consider lattice polytopes modulo unimodular equivalence, since such equivalence preserves, among other things, the lattice width.
\blue{This means that there is no loss of generality in assuming $\Lambda=\Z^d$, but in some examples it is more convenient to allow for a general $\Lambda$.}

If $p$ is a point in the affine span of a $d$-simplex $\Delta=\conv(v_0,\dots,v_d)$, the vector of \emph{barycentric coordinates} of $p$ with respect to $\Delta$ is the unique $(\alpha_0,\dots,\alpha_d) \in \R^{d+1}$ satisfying
\[
p=\sum_{i=0}^d \alpha_i v_i, \qquad 1=\sum_{i=0}^d \alpha_i.
\]

\blue{
If $\Delta = \conv(v_0,\dots,v_d)$ is a lattice simplex, we denote by $\Lambda_\Delta$ the lattice generated by $\{v_i-v_j : i,j \in\{0,\dots,d\}\}$. We denote $G_\Delta:= \Lambda/\Lambda_\Delta$ and call it the \emph{quotient group} of $\Delta$. It  is a finite abelian group of order equal to the \emph{normalized volume} of $\Delta$, that is, the ratio of the Euclidean volume of $\Delta$ to the Euclidean volume of a unimodular simplex in $\Lambda$. Throughout the paper, we always refer by \emph{volume} to the normalized volume. 

Observe that, modulo unimodular equivalence, there is no loss of generality in assuming $v_0=0$, in which case $\Lambda_\Delta$ is the linear lattice generated by vertices of $\Delta$. 
}

Each facet $F$ of $\Delta$ is itself a lattice simplex. As such, it has its own quotient group $G_F:=(\Lambda \cap \aff(F)) / \Lambda_F$, of order equal to the volume of $F$ (normalized to $\Lambda \cap \aff(F)$).

\begin{proposition}
\label{prop:notcyclic}
If $\Delta$ is a lattice $d$-simplex of volume $N\in \N$, then:
\begin{enumerate}
\item The barycentric coordinates of every lattice point $p\in \Lambda$ have the form
\[
\frac1N(b_0,\dots,b_d)
\]
where $b_0,\dots,b_d \in \Z$ satisfy $\sum_i b_i =N$
and each $b_i$ is a multiple of the normalized volume of the $i$-th facet (the facet where the $i$-th barycentric coordinate vanishes).

\item For each facet $F_i$ of $\Delta$, of volume $N_i$, we have a short exact sequence \[0\to G_{F_i} \to G_\Delta \to \Z_{N/N_i}\to 0.\]
\end{enumerate}
\end{proposition}

\begin{proof}
Let $\alpha=(\alpha_0, \ldots, \alpha_d)$ be the barycentric coordinates of $p$. Since $G_\Delta$ has order $N$, we have that $\alpha N$ is an integer vector, so that $b_i:=N \alpha_i$ is an integer. We have $\sum_i b_i =N$ because $\sum_i\alpha_i=1$.

Let $F_i$ be the $i$-th facet and $v_i$ the opposite vertex. \blue{For every affine functional $f$ vanishing on $F_i$ we have that $f(\Lambda)$ is an infinite cyclic subgroup of $\R$, so we can normalize it to have $f(\Lambda) = \Z$ and $f(v_i)>0$. We denote $f_i$ be the affine functional so obtained.} The volume $N$ of $\Delta$ equals the volume $N_i$ of $F_i$ times $f_i(v_i)$; hence $f_i(v_i) = N/N_i$. On the other hand, $f_i(p)$ is nothing but the $i$-th barycentric coordinate $\alpha_i$ times $f_i(v_i)$. Thus, in order to have $f_i(p)\in \Z$ we need $\alpha_i$ to be a multiple of $N_i/N$, hence $b_i$ be a multiple of $N_i$.

In part two, the first map is induced by the inclusion $\Lambda \cap \aff(F) \hookrightarrow \Lambda$, and the second map 
sends $\frac1N(b_0,\dots,b_d) \mapsto \frac{b_i}{N_i} \pmod {N/N_i}$. Both the image of the first map and the kernel of the second map consist of the lattice points with $b_i\equiv 0 \pmod N$.
The second map is surjective because (the class of) any $p\in \Lambda$ with $f_i(p)=1$ generates the image.
\end{proof}

\blue{
\begin{definition}
\label{defi:cyclic}
A lattice simplex $\Delta$ is  called \emph{cyclic} if the quotient group $G_\Delta$ is cyclic. 
\end{definition}
}

In what follows we show that a cyclic simplex $\Delta$ can be characterized modulo unimodular equivalence
by the vector of barycentric coordinates of a representative $p\in \R^d$ for a generator of $G_\Delta$. 
For more details about this see, e.~g.,~\cite{IglesiasSantos2020+}.

\begin{proposition}
\label{prop:cyclic}
If $\Delta$ is a cyclic simplex and $p$ is a generator of $G_\Delta$, with barycentric coordinates 
$
\frac1N(b_0,\dots,b_d)
$, then:

\begin{enumerate}
\item $\gcd(N,b_0,\dots,b_d)=1$.

\item The volume of the $i$-th facet of $\Delta$ equals $\gcd(N,b_i)$.
\end{enumerate}
\end{proposition}

\begin{proof}
The $\gcd$ has to be 1 in order for  $G_\Delta$ to have order $N$, since part (1) of Proposition~\ref{prop:notcyclic} is actually an embedding of $G_\Delta$ into $\frac1N\Z_N^{d+1}$.

For part (2): since $p$ generates $G_\Delta$, surjectivity in 
Proposition~\ref{prop:notcyclic}(2) implies that its image $\frac{b_i}{N_i} \pmod {\frac{N}{N_i}}$ is a generator of $\Z_{\frac{N}{N_i}}$, thus, $\gcd(\frac{b_i}{N_i}, \frac{N}{N_i})=1$.
\end{proof}

Two vectors $b$ and $b'$ of barycentric coordinates represent the same point in $\Lambda/\Lambda_\Delta$ if, and only if, they differ by an integer vector. This means that the entries $b_i$ are only important modulo $N$, as long as they add up to $N$. In practice, this last condition can be dropped if we remember the value of $N$. That is, a cyclic simplex can be characterized by its volume $N$ and a vector  $b\in \Z_N^{d+1}$ with $\sum b_i \equiv 0\pmod n$, representing the barycentric coordinates of a generator of $\Lambda/\Lambda_\Delta$. 
The following proposition expresses this, and it also gives an explicit coordinatization of every cyclic simplex.

\blue{
\begin{proposition}
\label{prop:cyclic-new}
Let $N\in \N$ and let $b=(b_0,\dots,b_d)\in \Z_N^{d+1}$ with $\sum_i b_i \equiv 0 \pmod N$ and  \(\gcd(N,b_0,\dots,b_d)=1\). 
Let $\Delta(N,b)$ be the standard simplex with vertices $0, e_1, \ldots, e_d$,
 but considered with respect to the lattice $\Lambda(N,b)$ generated by $\Z^d$ and the point $\frac1N (b_1,\dots,b_d)$.
 (Observe that now $\Z^d=\Lambda_{\Delta(N,b)}$). Then,  $\Delta(N,b)$ is a cyclic simplex of volume $N$. Moreover, every cyclic simplex is unimodularly equivalent to one of this form.
\end{proposition}

We call $\Delta(N,b)$, or any lattice simplex unimodularly equivalent to it, a \emph{cyclic $d$-simplex of volume $N$ and with generator $b$}.

\begin{proof}
By definition of $\Lambda$, the quotient group $G_{\Delta(N,b)}$ is cyclic with generator $\frac1N (b_1,\dots,b_d)$, and its size is $N$ thanks to the gcd condition. 

Conversely, given any cyclic simplex $\Delta$ with ambient lattice $\Lambda$ and vertex lattice $\Lambda_\Delta$, 
consider the affine map $f$ sending $v_0$ to the origin and $v_i$ to $e_i$. Let $p$ be a generator of the quotient group $G_\Delta$. By Proposition~\ref{prop:notcyclic} $p=\frac1N b$ for an integer vector $(b_0,\dots,b_d)$. If we denote $b=(b_0, b_1,\dots, b_d)$ we have that the sum of coordinates of $b$ is zero modulo $N$, 
$f(\Delta)$ is the standard simplex, and $f(\Lambda)$ equals $\Lambda(N,b)$.
\end{proof}
}

Our definition of $\Delta(N,b)$ seems to forget $b_0$, but it does not. Since $b_0$ is only important modulo $ N$ and 
$\sum_i b_i \equiv 0 \pmod N$, $b_0$ can be recovered from the rest. 


\begin{proposition}
\label{prop:cyclic2}
Let $N,N' \in \N$ and let $b=(b_0,\dots,b_d)$ and $b'=(b'_0,\dots,b'_d)$ be integer vectors. 
$\Delta(N,b)$ and $\Delta(N',b')$ (each considered to its respective lattice $\Lambda(N,b)$ and $\Lambda(N',b')$) are unimodularly equivalent if and only if $N=N'$ and $b'\pmod N$ can be obtained from $b\pmod N$ via multiplication by an integer scalar and permutation of coordinates.
\end{proposition}

\blue{
\begin{proof}
Sufficiency is clear, since $\Delta(N,b) = \Delta(N',b')$ is the standard simplex and the permutation of (barycentric) coordinates sending $\Lambda(N,b)$ to $\Lambda(N',b')$ is an affine map.

For necessity, let $f$ be the unimodular equivalence sending $\Delta(N,b)$ to $\Delta(N',b')$ and $\Lambda(N,b)$ to $\Lambda(N',b')$. Observe that $f$ is completely characterized by the bijection it induces from the vertices of $\Delta(N,b)$ to those of $\Delta(N',b')$. Put differently, $f$ is just a permutation $\sigma$ of barycentric coordinates. On the other hand, if  $p=\frac1N(b_0,\dots,b_n)$ and $p'=\frac1N(b'_0,\dots,b'_n)$ are the generators of the respective quotient groups (again written in barycentric coordinates), we have that, as elements of $G_{\Delta(N',b')}$,  $f(p)$ is a scalar multiple of $p'$. That is, there is a scalar factor $k$ such that $\frac{1}{N} \sigma(b) - \frac{k}{N} b'$ is an integer vector. That is, $\sigma(b) \equiv  k b' \pmod N$.
%
\end{proof}
}

\begin{example}
\label{exm:white}
As proved by White~\cite{White1964} (see also~\cite{Sebo1999}) every empty $3$-simplex is unimodularly equivalent to \[
T(p,N):= \conv \{ (0,0,0), (1,0,0), (0,0,1),(p,N,1) \}
\]
for some $N\in \N$ and $p\in \{1,\dots,N-1\}$, \blue{considered with respect to the integer lattice $\Z^3$.}
$T(p,N)$ is cyclic of volume $N$ and with generator $(N+p,-p,-1,1)$ since
\[
(0,1,0) =  \frac1N \left[ \left(N+p\right) (0,0,0) - p\,(1,0,0)   - (0,0,1) +  (p,N,1)\right].
\]
\end{example}

\begin{example}
\label{exm:4dim}
Barile et al.~\cite{BBBK11} proved that all empty $4$-simplices are cyclic. The enumeration by Iglesias and Santos shows that only 179 of them have width larger than two, namely the ones displayed in~\cite[Table 1]{IglesiasSantos2019}. Among these, the following are particularly interesting:
\begin{itemize}
\item The unique smallest one, of volume 41 and generator $(1,-4, 18, 16, 10)$.
\item The only one of width four (all the others have width three), of volume 101 and generator $(1,-6, -14, -17, 36)$.
\end{itemize}
\blue{We encourage the reader to check that in both cases the entries of $b$ are the fifth roots of unity modulo $N$.}
\end{example}

\section{Cyclotomic simplices}
\label{sec:cyclotomic}

Let $\Delta=\conv(v_0,\dots,v_d) \subset \R^d$ be a cyclic $d$-simplex of volume $N$ with generator $b=(b_0,\dots,b_d)$, and suppose that it has a {\em unimodular cyclic symmetry}, i.e., there is a unimodular affine transformation $f: \R^d\to \R^d$ sending $\Delta$ to itself and inducing a cyclic permutation of vertices. To make things concrete, assume that $f(v_i) = v_{i+1}$ for all $i$ (including $f(v_d) = v_0$; we consider the indices $0,\dots,d$ cyclically modulo $d+1$). 

\begin{lemma}
\label{lemma:cyclotomic}
A cyclic simplex $\Delta$ of a certain volume $N$ admits a unimodular cyclic symmetry if, and only if, it is unimodularly equivalent to the cyclic simplex with generator $(1,k,\dots,k^d)$ for some $k\in \Z_N$ satisfying 
\[
\sum_{i=0}^d k^i\equiv 0 \pmod N.
\]
\end{lemma}

\begin{proof}
Let $p$ be a generator of the quotient group $G_\Delta:=\Lambda/\Lambda_\Delta$ and let $\frac1N(b_0,\dots,b_d)$ be its barycentric coordinates.
Observe that every unimodular transformation preserving $\Delta$ induces an automorphism of the quotient group $G_\Delta:=\Lambda/\Lambda_\Delta$. Since we assume $G_\Delta \cong \Z_N$, such automorphisms equal multiplication by a unit $r\in \Z_N$. Hence:

If $f$ is a  unimodular cyclic symmetry, then the barycentric coordinates of $f(p)$ are $\frac1N(b_d,b_0,\dots,b_{d-1})$, so that $(b_d,b_0,\dots,b_{d-1}) \equiv r(b_0,\dots,b_{d}) \pmod N$. Put differently, $b_{i+1}\equiv kb_i\pmod N$ for all $i$ 
(including $b_0\equiv kb_d$), where $k$ is the inverse of $r$ modulo $N$.
Hence $b_i \equiv k^i b_0 \pmod N$ for all $i$.

Now, $b_0$ must be coprime with $N$, because if there is a prime $p$ dividing $\gcd(b_0, N)$ then $p$ divides every $b_i \equiv k^ib_0 \pmod N$ too, which contradicts Proposition~\ref{prop:cyclic}. 
Multiplication of  the vector $b$ by the inverse of $b_0$ modulo $N$ shows that $\Delta$ is genreated by $(1,k,\dots,k^d)$.

Conversely, suppose that $\Delta$ is generated by $(1,k,\dots,k^d)$. Then the affine map sending $v_i$ to $v_{i+1}$ for every $i$ sends the generator $p$ with barycentric coordinates $\frac1N(1,k,\dots,k^d)$ to a point $p'$ with barycentric coordinates $\frac1N(k^d, 1, k,\dots,k^{d-1})$. Since these coordinates equal $k^d$ times the coordinates of $p$ (because $1+k+\dots+k^d\equiv 0\pmod N$ implies $k^{d+1}\equiv 1 \pmod N$) we have that $p'$ is a lattice point in $\Lambda$. Hence, $f$ is a unimodular cyclic symmetry for $\Delta$.
\end{proof}

As mentioned in the proof, the condition $\sum_{i=0}^d k^i\equiv 0 \pmod N$ implies 
$k^{d+1}\equiv 1 \pmod N$. (The converse holds if $N$ is prime and $k\ne 1$, but not in general).
That is, the generator of a cyclic simplex with unimodular cyclic symmetry is the vector of powers of a \emph{$(d+1)$-th root of unity} modulo $N$. 
A priori we do not need $k$ to be a primitive root; but if $k$ is not primitive then the simplex has small width, as we now show:

\begin{corollary}
\label{coro:cyclotomic}
In the conditions of Lemma~\ref{lemma:cyclotomic}, 
\begin{enumerate}
\item If the width of $\Delta$ is more than two, then $k$ is a primitive $(d+1)$-th root of unity modulo $N$.

\item If $N$ is prime and the width of $\Delta$ is more than one, then $d+1$ is prime.
\end{enumerate}
\end{corollary}

\begin{proof}
If $k$ is a nonprimitive $(d+1)$-th root of unity modulo $N$ then the generating vector $(1,k,\dots,k^d)$ has repeated entries. Then, the functional taking values $+1$ and $-1$ in a pair of vertices corresponding to repeated entries, and zero in all other vertices, is an integer functional (because it has integer value in the generator $\frac1N(1,k,\dots,k^d)$ of the quotient group) and it gives width two to the simplex.

For part (2), assume that $N$ is prime but $d+1$ is not. Consider a factorization $d+1=mn$ with $m,n>1$, and let $k$ be a $(d+1)$-th root of unity modulo $N$, so that $\Delta$ is generated by $b=(1,k,k^2,\dots,k^d)$. 
Consider the affine functional $f$ that takes value $1$ at the $i$-th vertex if $i$ is a multiple of $m$ and $0$ otherwise.  The powers of $k^m$ in $b$ add up to zero, because they are the roots of the polynomial $x^n-1$. Thus, $f$ takes an integer value in the generator $\frac1N(1,k,\dots,k^d)$ of the quotient group and hence, it is an integer functional giving width one to the simplex.

Remark: in both parts we use the following criterion for a functional $f$, expressed as a vector in $\Z^{d+1}$ that lists
the values of $f$ on vertices of a cyclic simplex, to be an integer functional: $f$ is integer if and only if  $f\cdot b \equiv 0\pmod N$, where $b$ is the generator of $\Delta$ as a cyclic simplex. This will be explained in more detail after Algorithm~\ref{alg:width}.
\end{proof}

In what follows we assume, for simplicity, that $N$ is prime. Hence, the multiplicative group of $\Z_N$ is cyclic of order $N-1$, and $(d+1)$-th roots of unity exist if and only if $d+1$ divides $N-1$. In this case, there are exactly $d+1$ such roots, and they are the powers of a primitive root $k$. Moreover, since $(d+1)$-th roots are the zeroes of the polynomial  $x^{d+1}-1 = (x-1)(1+x+\dots+x^d)$, $k$ satisfies the extra condition $\sum_{i=0}^d k^i\equiv 0 \pmod N$ of Lemma~\ref{lemma:cyclotomic}.

\begin{definition}
\label{defi:cyclotomic}
Let $d,N$ be two positive integers with $N$ prime and $d+1$ a divisor of $N-1$.
Let $a_0,\dots,a_d$ be the $(d+1)$-th roots of unity modulo $N$, that is, the powers of any primitive $(d+1)$-th root.

We call \emph{cyclotomic simplex} of parameters $(d,N)$ the cyclic simplex of volume $N$ with generator
$(a_0,\dots,a_d)$. We  denote it $\Cyc(d,N)$.
\end{definition}

%

\begin{example}
\label{exm:4dim_cont}
The empty $4$-simplices $(1,-4, 18, 16,10)$ and $(1,-6, -14, -17, 36)$ of Example~\ref{exm:4dim} are cyclotomic. Indeed, they have volumes 41 and 101 and we have that 
\centerline{$(1,-4, 18, 16,10) \equiv (1,-4, -4^3, 4^2,4^4) \pmod{41}$, with $-4^5\equiv 1 \pmod{41}$.}
\centerline{$(1,-6, -14, -17, 36)\equiv (1,-6, -6^3, 6^4,6^2) \pmod{101}$, with $-6^5\equiv 1 \pmod{101}$.}

The fact that the smallest empty 4-simplex of width three and the unique empty 4-simplex of width four are cyclotomic seems to indicate that, although quite specific, the cyclotomic construction is useful in order to find wide empty simplices.
\end{example}

The following statement summarizes Lemma \ref{lemma:cyclotomic} and Corollary \ref{coro:cyclotomic} in the prime \(N\) setting:

\begin{corollary}
\label{coro:cyclotomic2}
Let $\Delta$ be a cyclic $d$-simplex of prime volume $N$. $\Delta$ admits a unimodular cyclic symmetry if and only if  $d+1$ divides $N-1$ and $\Delta$ is unimodularly equivalent to the cyclotomic simplex $\Cyc(d,N)$. 

For $\Delta$ to have width more than one, $d+1$ must be prime, too.
\end{corollary}

\begin{remark}
\label{rem:primes}
Strictly speaking, the simplices we are interested in can be defined via the existence of a \emph{principal root of unity} (a root of unity with the property $1+k+\dots+k^d \equiv 0 \pmod N$ of Lemma~\ref{lemma:cyclotomic}) without the primality assumption on $N$ or $d+1$. 
Since we want the simplex to be wide, we also want $k$ to be a primitive root of unity (Corollary~\ref{coro:cyclotomic}). These conditions imply that $d+1$ divides the Euler value $\phi(N)$, but the converse is not true. 

Consider for example $N=15$, so that  $\phi(N)=8$. We have that: (a) Even though $8$ divides $\phi(N)$, there are no primitive $8$-th roots, since the group of units modulo 15 is isomorphic to $\Z_4\oplus \Z_2$. (b) There are $4$-th primitive roots, and they come in two orbits. The powers of $-2$ do not add up to 0 (modulo 15) but the powers of $2$ do.

Thus, we could extend our definition of cyclotomic simplices to admit the existence of $\Cyc(4,15)$ taking $k=2$. But if we admit $N$ not to be prime it is no longer true that  $\Cyc(d,N)$ is unique, whenever it exists. See Remark~\ref{rem:primes2}.
%
%
\end{remark}

\subsection*{Algorithmic remarks}

We have computationally explored cyclotomic $d$-simplices for $d=4, 6, 10$. 
Our method is:
\begin{enumerate}
\item We go through all the primes $N$ with $N-1$ a multiple of $d+1$, in a range that we set appropriately in each case.

\item For each such $N$, we find the $(d+1)$-th roots of unity modulo $N$ and construct the cyclotomic simplex $\Cyc(d,N)$.

\item We check whether $\Cyc(d,N)$ is empty and compute its width.
\end{enumerate}

Steps (1) and (2) are computationally straightforward. (In step (2) we can use brute force since step (3) is much more expensive). Let us discuss how we check emptyness and compute width of a cyclic simplex:

\subsubsection*{Checking emptyness} In barycentric coordinates dilated by $N$, all lattice points are multiples modulo $N$ of the generator $b=(b_0,\dots,b_d)$. If such a multiple $j b\in \Z_N^{d+1}$ has a representative that lies in the simplex, then that representative has all barycentric coordinates in $\{0,\dots,N-1\}$. 
Equivalently, the unique representative of $j b\in \Z_N^{d+1}$ with coordinates  in $\{0,\dots,N-1\}$ must have sum of coordinates equal to $N$, so that it represents the barycentric coordinates of a point. This implies:

\begin{proposition}
\label{prop:emptyness}
The cyclic simplex of volume $N$ and generator $b=(b_0,\dots,b_d)\in \Z_N^{d+1}$ is empty if, and only if, for every $j \in \{1,\dots, N-1\}$ reducing $j b$ modulo $N$ so that it lies in $\{0,\dots,N-1\}^{d+1}$ yields a vector whose sum of coordinates differs from $N$.
\end{proposition}

\subsubsection*{Computing width}
Computing the width of a general lattice simplex can be formulated as an integer linear minimization problem. For simplices, however, a simple exhaustive algorithm is as follows:

\begin{algorithm} \ 
\label{alg:width}

\noindent INPUT:  a cyclic $d$-simplex $\Delta$ given by its generator $b=(b_0,\dots,b_d)\in \Z^{d+1}$ and its volume $N$, plus a target width $w\in \N$.

\noindent
OUTPUT: if $\Delta$ has lattice width $\le w$, an integer linear functional attaining that width. If not, the information that such a functional does not exist.

Step 1) For each $f\in \{0,\ldots,w\}^{d+1}$, compute $f\cdot b $. If $f\cdot b \in N\Z$ then add $f$ to the list of functionals giving width $\le w$ to $\Delta$ and keep $\width_f(\Delta) = \max\{f_i\} - \min\{f_i\}$.

Step 2) If a functional is found in Step 1, return one of minimum width. If not, report that the width of $\Delta$ is greater than $w$.
\end{algorithm}

The algorithm works because of the following: $\Delta$ has lattice width $\le w$ if, and only if, there is an integer affine functional taking values in $\{0,\ldots,w\}$ at all the vertices of $\Delta$. In the algorithm, we represent each affine functional $f$ via its list of values at the vertices of $\Delta$, and look at all those (integral or not), with values in $\{0,\ldots,w\}$. Each such functional is integer (on the affine lattice $\Lambda$ generated by our vertices and the point of barycentric coordinates $\frac1Nb$) if, and only if, its value at the generator $\frac1Nb$ is an integer; that is, if its value at $b$ is in $N\Z$.

This algorithm is very expensive (there are $(w+1)^{d+1}$ functionals to check, and we typically set $w \sim d+1$) but we use several tricks to speed it up:

\begin{itemize}

\item We can assume that $f$ has some entry equal to 0.  
By cyclic symmetry of our simplex, there is no loss of generality in assuming that entry to be  $f_0$, so that we only iterate over $d$ entries of $f$, instead of $d+1$.

\item For the other $d$ coordinates, we use ``bidirectional search'' in order to reduce our search space from size $(w+1)^d$ to $(w+1)^{d/2}$.%
\footnote{Our $d$ is always even. If it were not, the same trick works by splitting the $d$ coordinates into $\lceil d/2\rceil + \lfloor d/2\rfloor$.}
For this, each functional $f\in  \{0,\ldots,w\}^{d}$ is decomposed as a sum $f_1 + f_2$ with $f_i \in \{0,\ldots,w\}^{d/2}$, where $f_1$ acts on the first $d/2$ entries of $b$ and $f_2$ and the last $d/2$. We compute two lists $L_1$ and $L_2$, with $L_i$ containing the scalar products of all the $f_i$ with the corresponding half of $b$, and we then check whether there is a value in $L_1$ and a value in $L_2$ that add up to $0$ modulo $N$. The latter can be done in time proportional to $\max_{i=1,2} |L_i| \log |L_i| \in O((w+1)^{d/2}d\log w)$ by first sorting the lists $L_1$ and $L_2$.

\item As a practical matter of implementation, we can give up finding a representative \(f\) and instead only output the existence or non-existence of some functional giving width \(w\). This allows parallel computation of the \(L_i\) by keeping a list of all values of the form \(\sum_{i=1}^k a_if_i \pmod N\) with \(k\) incrementing from \(1\). In each step, only the previous list needs to be stored, and to each value in this list, \(w\) additions need to be performed. These lists can be kept sorted, at cost linear in their length. This slightly improves the complexity to \( O((w+1)^{d/2})\), which is both the sorting cost, and the count of additions performed when \(k=d/2\).
\end{itemize}

\subsection*{Experimental results}

We have explored cyclotomic simplices for $d=4,6,10$ with prime $N$ up to $2^{31} \simeq 2.14{\cdot} 10^9$, finding:

\begin{itemize}

\item For $d=4$ there are only four empty cyclotomic simplices: the two from Example~\ref{exm:4dim_cont} plus another two, of volumes 11 and 61. They are described in Table~\ref{table:44}.

\begin{table}
\small
\[
\begin{array}{c|ccc}
\text{simplex}& N \text{ (volume)} & b\pmod N & \text{width}\\
\hline
\Cyc(4, 11) &11 & (1, -2, 4, -8, 5) & 2 \\
\Cyc(4, 41) &41 & (1,-4, 16, 18, 10) &3 \\
\Cyc(4, 61) &61 & (1,-3, 9, -27, 20)  & 3\\
\Cyc(4, 101) &101 & (1,-6, 36, -14, -17) & 4\\
\end{array}
\]
\caption{The four empty cyclotomic 4-simplices. $\Cyc(4, 101)$ is the unique empty $4$-simplex of maximum volume (see~\cite{IglesiasSantos2019}).}
\label{table:44}
\end{table}

\item For $d=6$ there are 88 empty cyclotomic simplices. Their volumes range from 
29 to 17683. Their maximum lattice width is 6, attained by the six simplices in Table~\ref{table:666}.
Table~\ref{table:6stats} shows additional statistics of cyclotomic $6$-simplices, both empty and non-empty.

\begin{table}
\small
\[
\begin{array}{c|ccc}
\text{simplex}& N \text{ (volume)} & b\pmod N \\
\hline
\Cyc(6, 6301) &6301 & (1, 4073, 5097, 4587, 386, 3229, 1530) \\
\Cyc(6, 10753) &10753 & (1, 8246, 5297, 376, 3632, 2367, 1587)  \\
\Cyc(6, 11117) &11117 & (1, 6165, 9319, 10096, 8874, 1453, 8560)  \\
\Cyc(6, 15121) &15121 & (1, 9543, 10187, 1632, 14667, 7205, 2128)  \\
\Cyc(6, 16493) &16493 & (1, 3665, 6923, 6561, 15764, 81, 16484)  \\
\Cyc(6, 17683) &17683 & (1, 12135, 11884, 7475, 13018, 11191, 15028)  \\
\end{array}
\]
\caption{The six empty cyclotomic 6-simplices of width 6. As mentioned in Remark~\ref{rem:primes2}, extending the search to non-prime volumes we have found a seventh empty $6$-simplex of width six, with volume $6\,931$.
}
\label{table:666}
\end{table}

\begin{table}
\small
\[
\begin{array}{c|ccc}
N& \text{\# empty} & \text{\# not-empty}  \\
\hline
\ \ \ \ \ \ [0,\ \ 1999] & 40  & 5    \\    
\ \ [2000,\ \ 3999] &  21 & 23    \\  
\ \ [4000,\ \ 5999] &  10 & 27    \\  
\ \ [6000,\ \ 7999] &  8 & 31    \\ %
\ \ [8000,\ \ 9999] &  3 & 35   \\  %
\ [10000,11999] & 2  & 37    \\ 
\ [12000,13999] & 1 &  29  \\ 
\ [14000,15999] & 1 &  34  \\ 
\ [16000,17999] & 2  & 32    \\ 
\end{array}
\qquad\qquad
\begin{array}{c|ccc}
\text{width} &\text{simplex}& \text{empty?}\\
\hline
1& \text{none}  & \text{---}       \\
2&\Cyc(6, 29) &  \text{yes}      \\  
3&\Cyc(6, 127) &  \text{no}      \\ 
4&\Cyc(6, 701) & \text{yes}     \\ 
5&\Cyc(6, 3347) & \text{no}     \\    
6&\Cyc(6, 6301) & \text{yes}       \\ 
7&\Cyc(6, 14197) & \text{no}      \\ 
8&\Cyc(6, 32369) & \text{no}      \\  
\end{array}
\]
\caption{Statistics on cyclotomic 6-simplices. Left: number of empty and non-empty cyclotomic simplices for each interval of volume. Right: smallest cyclotomic simplex for each width}
\label{table:6stats}
\end{table}

\item For $d=10$ we have found $218\,075$ empty cyclotomic simplices. Their number decreases very fast with volume, as seen in Figure~\ref{fig:bars}. For example, there are only 30 of volume above $10^9$ and only one above $15{\cdot}10^8$, of volume $1\,757\,211\,061$ and width 10.
There are none of  width larger than 11, and exactly the following five of width 11:
\[
\begin{array}{cc}
\Cyc(10, \ 582\, 595\, 883), &
\Cyc(10, \ 728\, 807\, 201),  \\
\Cyc(10, \ 976\, 965\, 023),  &
\Cyc(10, \ 1\, 066\, 142\, 419), \\
\Cyc(10, \ 1\, 113\, 718\, 783).
\end{array}
\]
Table~\ref{table:10} shows the smallest empty cyclotomic $10$-simplex of width $w$ for each $w\in 1,\dots,11$

\begin{table}
\small
\[
\begin{array}{cc|c}
\text{width}&&   \text{simplex}\\
\hline
1&&\Cyc(10, 23)  \\
2&&\Cyc(10, 199)  \\
3&&\Cyc(10, 4\, 159)  \\
4&&\Cyc(10, 55\, 243)  \\
5&&\Cyc(10, 237\, 161)  \\
6&&\Cyc(10, 1\, 197\, 901)  \\
7&(n.e.)&\Cyc(10, 5\, 631\, 979) \\
7&(e.) &\Cyc(10, 5\, 989\, 699)  \\
\\
\end{array}
\qquad
\begin{array}{cc|c}
\text{width} && \text{simplex}\\
\hline
8&&\Cyc(10, 15\, 465\, 649)  \\
9&(n.e.)&\Cyc(10, 56\, 870\, 903)  \\
9&(e.)&\Cyc(10, 60\, 822\, 851)  \\
10&(n.e.)&\Cyc(10, 155\, 176\, 451)   \\
10&(e.)&\Cyc(10, 176\, 796\, 313)  \\
11&(n.e.)&\Cyc(10, 313\, 968\, 931)  \\
11&(e.)&\Cyc(10, 582\, 595\, 883)  \\
\\ \\
\end{array}
\]
\caption{Smallest cyclotomic 10-simplex for each width from 1 to 11. The smallest cyclotomic simplex of each width is also empty for widths up to 6 and equal to 8, but not for widths $7,9,10,11$. For the latter we list both the absolute smallest ``(\emph{n.e.})'' and the smallest empty one ``(\emph{e.})''.}
\label{table:10}
\end{table}

\begin{figure}
\begin{tikzpicture}[xscale=0.05,yscale=0.5]
\draw (-1.5,{log2(0.1)}) rectangle (0.5,{log2(0.1)});
\draw (-1.5,{log2(0.2)}) rectangle (0.5,{log2(0.2)});
\draw (-1.5,{log2(0.5)}) rectangle (0.5,{log2(0.5)});
\draw (-3.5,{log2(1)}) rectangle (0.5,{log2(1)});
\draw (-1.5,{log2(2)}) rectangle (0.5,{log2(2)});
\draw (-1.5,{log2(5)}) rectangle (0.5,{log2(5)});
\draw (-3.5,{log2(10)}) rectangle (0.5,{log2(10)});
\draw (-1.5,{log2(20)}) rectangle (0.5,{log2(20)});
\draw (-1.5,{log2(50)}) rectangle (0.5,{log2(50)});
\draw (-3.5,{log2(100)}) rectangle (0.5,{log2(100)});
\draw (-1.5,{log2(200)}) rectangle (0.5,{log2(200)});
\draw (-1.5,{log2(500)}) rectangle (0.5,{log2(500)});
\draw (-3.5,{log2(1000)}) rectangle (0.5,{log2(1000)});
\draw (-1.5,{log2(2000)}) rectangle (0.5,{log2(2000)});
\draw (-1.5,{log2(5000)}) rectangle (0.5,{log2(5000)});
\draw (-3.5,{log2(10000)}) rectangle (0.5,{log2(10000)});
\draw (0.5,{log2(0.05)}) rectangle (0.5,{log2(11000)});
\node at (-5.5,{log2(0.1)}) {$1$};
\node at (-5.5,{log2(0.2)}) {$2$};
\node at (-5.5,{log2(0.5)}) {$5$};
\node at (-7.5,{log2(1)}) {$10$};
\node at (-7.5,{log2(2)}) {$20$};
\node at (-7.5,{log2(5)}) {$50$};
\node at (-9,{log2(10)}) {$100$};
\node at (-9,{log2(100)}) {$10^3$};
\node at (-9,{log2(1000)}) {$10^4$};
\node at (-9,{log2(10000)}) {$10^5$};
\draw (0.5,{log2(0.05)}) rectangle (1.5,{log2(5957.2)});
\draw (1.5,{log2(0.05)}) rectangle (2.5,{log2(4164.0)});
\draw (2.5,{log2(0.05)}) rectangle (3.5,{log2(3036.9)});
\draw (3.5,{log2(0.05)}) rectangle (4.5,{log2(2228.6)});
\draw (4.5,{log2(0.05)}) rectangle (5.5,{log2(1661.0)});
\draw (5.5,{log2(0.05)}) rectangle (6.5,{log2(1211.5)});
\draw (6.5,{log2(0.05)}) rectangle (7.5,{log2(892.2)});
\draw (7.5,{log2(0.05)}) rectangle (8.5,{log2(635.4)});
\draw (8.5,{log2(0.05)}) rectangle (9.5,{log2(473.5)});
\draw (9.5,{log2(0.05)}) rectangle (10.5,{log2(350.3)});
\draw (10.5,{log2(0.05)}) rectangle (11.5,{log2(266.2)});
\draw (11.5,{log2(0.05)}) rectangle (12.5,{log2(197.2)});
\draw (12.5,{log2(0.05)}) rectangle (13.5,{log2(153.2)});
\draw (13.5,{log2(0.05)}) rectangle (14.5,{log2(106.5)});
\draw (14.5,{log2(0.05)}) rectangle (15.5,{log2(77.6)});
\draw (15.5,{log2(0.05)}) rectangle (16.5,{log2(62.8)});
\draw (16.5,{log2(0.05)}) rectangle (17.5,{log2(50.3)});
\draw (17.5,{log2(0.05)}) rectangle (18.5,{log2(41.8)});
\draw (18.5,{log2(0.05)}) rectangle (19.5,{log2(32.3)});
\draw (19.5,{log2(0.05)}) rectangle (20.5,{log2(22.9)});
\draw (20.5,{log2(0.05)}) rectangle (21.5,{log2(18.4)});
\draw (21.5,{log2(0.05)}) rectangle (22.5,{log2(17.3)});
\draw (22.5,{log2(0.05)}) rectangle (23.5,{log2(15.4)});
\draw (23.5,{log2(0.05)}) rectangle (24.5,{log2(11.2)});
\draw (24.5,{log2(0.05)}) rectangle (25.5,{log2(13.5)});
\draw (25.5,{log2(0.05)}) rectangle (26.5,{log2(10.5)});
\draw (26.5,{log2(0.05)}) rectangle (27.5,{log2(8.2)});
\draw (27.5,{log2(0.05)}) rectangle (28.5,{log2(8.7)});
\draw (28.5,{log2(0.05)}) rectangle (29.5,{log2(6.1)});
\draw (29.5,{log2(0.05)}) rectangle (30.5,{log2(5.5)});
\draw (30.5,{log2(0.05)}) rectangle (31.5,{log2(7.8)});
\draw (31.5,{log2(0.05)}) rectangle (32.5,{log2(5.1)});
\draw (32.5,{log2(0.05)}) rectangle (33.5,{log2(4.7)});
\draw (33.5,{log2(0.05)}) rectangle (34.5,{log2(4.2)});
\draw (34.5,{log2(0.05)}) rectangle (35.5,{log2(2.9)});
\draw (35.5,{log2(0.05)}) rectangle (36.5,{log2(2.6)});
\draw (36.5,{log2(0.05)}) rectangle (37.5,{log2(3.3)});
\draw (37.5,{log2(0.05)}) rectangle (38.5,{log2(2.5)});
\draw (38.5,{log2(0.05)}) rectangle (39.5,{log2(2.2)});
\draw (39.5,{log2(0.05)}) rectangle (40.5,{log2(1.6)});
\draw (40.5,{log2(0.05)}) rectangle (41.5,{log2(2.1)});
\draw (41.5,{log2(0.05)}) rectangle (42.5,{log2(0.6)});
\draw (42.5,{log2(0.05)}) rectangle (43.5,{log2(1.4)});
\draw (43.5,{log2(0.05)}) rectangle (44.5,{log2(1.4)});
\draw (44.5,{log2(0.05)}) rectangle (45.5,{log2(0.6)});
\draw (45.5,{log2(0.05)}) rectangle (46.5,{log2(0.5)});
\draw (46.5,{log2(0.05)}) rectangle (47.5,{log2(0.5)});
\draw (47.5,{log2(0.05)}) rectangle (48.5,{log2(0.5)});
\draw (48.5,{log2(0.05)}) rectangle (49.5,{log2(0.4)});
\draw (49.5,{log2(0.05)}) rectangle (50.5,{log2(0.4)});
\draw (50.5,{log2(0.05)}) rectangle (51.5,{log2(1.2)});
\draw (51.5,{log2(0.05)}) rectangle (52.5,{log2(0.4)});
\draw (52.5,{log2(0.05)}) rectangle (53.5,{log2(0.9)});
\draw (53.5,{log2(0.05)}) rectangle (54.5,{log2(0.2)});
\draw (54.5,{log2(0.05)}) rectangle (55.5,{log2(0.5)});
\draw (55.5,{log2(0.05)}) rectangle (56.5,{log2(0.4)});
\draw (56.5,{log2(0.05)}) rectangle (57.5,{log2(0.2)});
\draw (57.5,{log2(0.05)}) rectangle (58.5,{log2(0.4)});
\draw (58.5,{log2(0.05)}) rectangle (59.5,{log2(0.5)});
\draw (59.5,{log2(0.05)}) rectangle (60.5,{log2(0.5)});
\draw (60.5,{log2(0.05)}) rectangle (61.5,{log2(0.4)});
\draw (62.5,{log2(0.05)}) rectangle (63.5,{log2(0.4)});
\draw (64.5,{log2(0.05)}) rectangle (65.5,{log2(0.3)});
\draw (65.5,{log2(0.05)}) rectangle (66.5,{log2(0.6)});
\draw (67.5,{log2(0.05)}) rectangle (68.5,{log2(0.2)});
\draw (68.5,{log2(0.05)}) rectangle (69.5,{log2(0.2)});
\draw (70.5,{log2(0.05)}) rectangle (71.5,{log2(0.2)});
\draw (71.5,{log2(0.05)}) rectangle (72.5,{log2(0.3)});
\draw (72.5,{log2(0.05)}) rectangle (73.5,{log2(0.6)});
\draw (73.5,{log2(0.05)}) rectangle (74.5,{log2(0.4)});
\draw (74.5,{log2(0.05)}) rectangle (75.5,{log2(0.4)});
\draw (75.5,{log2(0.05)}) rectangle (76.5,{log2(0.2)});
\draw (77.5,{log2(0.05)}) rectangle (78.5,{log2(0.1)});
\draw (78.5,{log2(0.05)}) rectangle (79.5,{log2(0.3)});
\draw (81.5,{log2(0.05)}) rectangle (82.5,{log2(0.2)});
\draw (82.5,{log2(0.05)}) rectangle (83.5,{log2(0.1)});
\draw (83.5,{log2(0.05)}) rectangle (84.5,{log2(0.3)});
\draw (85.5,{log2(0.05)}) rectangle (86.5,{log2(0.3)});
\draw (86.5,{log2(0.05)}) rectangle (87.5,{log2(0.4)});
\draw (87.5,{log2(0.05)}) rectangle (88.5,{log2(0.1)});
\draw (88.5,{log2(0.05)}) rectangle (89.5,{log2(0.1)});
\draw (89.5,{log2(0.05)}) rectangle (90.5,{log2(0.1)});
\draw (90.5,{log2(0.05)}) rectangle (91.5,{log2(0.2)});
\draw (92.5,{log2(0.05)}) rectangle (93.5,{log2(0.2)});
\draw (93.5,{log2(0.05)}) rectangle (94.5,{log2(0.2)});
\draw (94.5,{log2(0.05)}) rectangle (95.5,{log2(0.1)});
\draw (95.5,{log2(0.05)}) rectangle (96.5,{log2(0.4)});
\draw (97.5,{log2(0.05)}) rectangle (98.5,{log2(0.2)});
\draw (99.5,{log2(0.05)}) rectangle (100.5,{log2(0.3)});
\draw (100.5,{log2(0.05)}) rectangle (101.5,{log2(0.2)});
\draw (102.5,{log2(0.05)}) rectangle (103.5,{log2(0.1)});
\draw (103.5,{log2(0.05)}) rectangle (104.5,{log2(0.4)});
\draw (104.5,{log2(0.05)}) rectangle (105.5,{log2(0.3)});
\draw (105.5,{log2(0.05)}) rectangle (106.5,{log2(0.2)});
\draw (106.5,{log2(0.05)}) rectangle (107.5,{log2(0.2)});
\draw (107.5,{log2(0.05)}) rectangle (108.5,{log2(0.2)});
\draw (110.5,{log2(0.05)}) rectangle (111.5,{log2(0.1)});
\draw (111.5,{log2(0.05)}) rectangle (112.5,{log2(0.2)});
\draw (115.5,{log2(0.05)}) rectangle (116.5,{log2(0.2)});
\draw (117.5,{log2(0.05)}) rectangle (118.5,{log2(0.2)});
\draw (118.5,{log2(0.05)}) rectangle (119.5,{log2(0.1)});
\draw (121.5,{log2(0.05)}) rectangle (122.5,{log2(0.1)});
\draw (123.5,{log2(0.05)}) rectangle (124.5,{log2(0.1)});
\draw (130.5,{log2(0.05)}) rectangle (131.5,{log2(0.1)});
\draw (142.5,{log2(0.05)}) rectangle (143.5,{log2(0.1)});
\draw (147.5,{log2(0.05)}) rectangle (148.5,{log2(0.1)});
\draw (175.5,{log2(0.05)}) rectangle (176.5,{log2(0.1)});
\node at (1,-4.9) {$0$};
\node at (51,-4.9) {$.5$};
\node at (101,-4.9) {$1.0$};
\node at (151,-4.9) {$1.5$};
\node at (201,-4.9) {$2.0$};
\node at (215,-4.9) {$\times 10^9$};
\draw (0.5,{log2(0.038)}) rectangle (0.5,{log2(0.05)});
\draw (10.5,{log2(0.045)}) rectangle (10.5,{log2(0.05)});
\draw (20.5,{log2(0.045)}) rectangle (20.5,{log2(0.05)});
\draw (30.5,{log2(0.045)}) rectangle (30.5,{log2(0.05)});
\draw (40.5,{log2(0.045)}) rectangle (40.5,{log2(0.05)});
\draw (50.5,{log2(0.038)}) rectangle (50.5,{log2(0.05)});
\draw (60.5,{log2(0.045)}) rectangle (60.5,{log2(0.05)});
\draw (70.5,{log2(0.045)}) rectangle (70.5,{log2(0.05)});
\draw (80.5,{log2(0.045)}) rectangle (80.5,{log2(0.05)});
\draw (90.5,{log2(0.045)}) rectangle (90.5,{log2(0.05)});
\draw (100.5,{log2(0.038)}) rectangle (100.5,{log2(0.05)});
\draw (110.5,{log2(0.045)}) rectangle (110.5,{log2(0.05)});
\draw (120.5,{log2(0.045)}) rectangle (120.5,{log2(0.05)});
\draw (130.5,{log2(0.045)}) rectangle (130.5,{log2(0.05)});
\draw (140.5,{log2(0.045)}) rectangle (140.5,{log2(0.05)});
\draw (150.5,{log2(0.038)}) rectangle (150.5,{log2(0.05)});
\draw (160.5,{log2(0.045)}) rectangle (160.5,{log2(0.05)});
\draw (170.5,{log2(0.045)}) rectangle (170.5,{log2(0.05)});
\draw (180.5,{log2(0.045)}) rectangle (180.5,{log2(0.05)});
\draw (190.5,{log2(0.045)}) rectangle (190.5,{log2(0.05)});
\draw (200.5,{log2(0.038)}) rectangle (200.5,{log2(0.05)});
\draw (0.5,{log2(0.05)}) rectangle (210,{log2(0.05)});
\end{tikzpicture}
\caption{The counts of empty cyclotomic 10-simplices with volumes up to $2^{31}$. Each bar counts the simplices whose volume falls within a block of 10 million. The vertical scale is logarithmic.}
\label{fig:bars}
\end{figure}

\begin{remark}
\label{rem:primes2}
For $d=4,6,10$ we have also explored what happens if we allow $N$ not to be prime.
This expanded search is significantly more computationally expensive because, as mentioned in Remark~\ref{rem:primes}, roots of unity may come in several orbits for each $N$ and they all have to be considered; in particular, for cyclotomic $d$ simplices with non-prime volume $N$ we use the notation $\Cyc(d,N,r)$, where $r$ is the root of unity producing that simplex. Yet, we have the following necessary condition for emptyness: for some $\Cyc(d,N,r)$ to be empty one needs $\Cyc(d,p)$ empty for all prime divisors $p$ of $N$.

In dimension four no extra empty cyclotomic simplices arise, but in dimension six we get 28 new ones, besides the 88 simplices of prime volume.
Their volumes range from 841 to 13\,021 and are always the product of two primes from $\{29, 43, 71, 113, 197, 211, 239$, $337, 449\}$.  
Among them there is one of width 6, with volume $N= 6931$. 
Non-prime volumes often produce several orbits of 7th roots of unity. E.g., for 6931 there are six orbits, which produce three non-empty simplices, two empty simplices of width 4, and one empty simplex of width 6.

In dimension ten, we only considered simplices of volume less than \(10^{10}\) and with two factors in their volume.
This search resulted in many additional
empty simplices, with volume ranging between 529 and 3\,637\,493\,861.
In this set, three additional empty simplices were found with width 11:
\begin{align*}
\Cyc(10, 476\, 164\, 613, 403\, 772\, 212), 
\\
\Cyc(10, 702\, 431\, 819, 305\, 486\, 391),
\\
\Cyc(10, 1\, 419\, 547\, 823, 822\, 028\, 992). 
\end{align*}
Observe that these include 10-simplices of width 11 both with smaller and bigger volumes than all the ones obtained with prime volumes.
\end{remark}

\end{itemize}

\section{Circulant simplices}
\label{sec:circulant}

\subsection{Circulant simplices, and our specific family \texorpdfstring{$\A{d}{m}$}{S(d,m)}}

A second way of guaranteeing that a $d$-simplex has a unimodular cyclic symmetry acting on its vertices is by permuting coordinates. 
Fix a dimension $d$, and let $v=(v_0,\dots,v_d)\in \R^{d+1}$ be any real vector. For each $i\in \N$ let $v^{(i)}$ be the vector obtained from $v$ by a cyclic shift (to the right) of $i$ units in its coordinates. 
That is $v^{(d+1)} = v^{(0)}=v$ and for $i =1, \ldots, d$:
\[
v^{(i)}:=(v_{d+1-i}, \dots,v_d,v_0,\dots,v_{d-i})
\]

\begin{definition}
\label{def:circulant}
If the $v^{(i)}$'s defined above are linearly independent, we call 
\[
\conv(v^{(0)},\dots, v^{(d)})
\]
the \emph{circulant simplex} (of dimension $d$) with generator $v$. 
\end{definition}

The name comes from the fact that \blue{square matrices in which all columns are obtained from the first one by cyclic shifts are called circulant, and} the matrix $(\vi[0] \dots \vi[d])$ with columns the vertex coordinates of a circulant simplex is circulant.
Linear independence requires that $\sum_i v_i \ne 0$. In fact, since the circulant simplex with generator $v$
affinely spans the hyperplane $\sum_i x_i = \sum_i v_i$, its volume equals
\begin{align}
\label{eq:volume}
\Vol(\conv(\vi[0] \dots \vi[d]))=
\left|
\frac{\det(\vi[0] \dots \vi[d])}{ \sum_{i=0}^d v_i}
\right|.
\end{align}

The example of interest to us is the following. 
For each $m\in \R$ and $d\in \N$, we denote $\A{d}{m}$ the circulant simplex of dimension $d$ and generator $(1,m,0,\dots,0,-m)$. That is, the vertices of $\A{d}{m}$ are the columns of the following $(d+1)\times(d+1)$ circulant matrix:
\[
M(d,m):=
\begin{pmatrix}
	1      & -m     & 0      & \ldots & \ldots & 0      & m      \\
	m      & 1      & -m     & \ddots &        &        & 0      \\
	0      & m      & 1      & \ddots & \ddots &        & \vdots \\
	\vdots & \ddots & \ddots & \ddots & \ddots & \ddots & \vdots \\
	\vdots &        & \ddots & \ddots & 1      & -m     & 0      \\
	0      &        &        & \ddots & m      & 1      & -m     \\
	-m     & 0      & \ldots & \ldots & 0      & m      & 1
\end{pmatrix}
\]
For odd $d$, the width of $\A{d}{m}$ is $1$ with respect to the functional $\sum_{i=0}^{(d-1)/2} x_{2i}$, which alternates values 0 and 1 on  vertices. Thus, we will always assume $d$ even. Since we are interested in lattice simplices we will sometimes assume $m$ to be an integer, but in some proofs we need to regard $m$ as a continuous parameter.

\begin{remark}
Although this family looks too specific to deserve attention per se, we were led to studying it from the (not at all obvious) realization that the unique empty $4$-simplex of width 4 is unimodularly equivalent to $\A{4}{2}$. Moreover, $\A{6}{3}$ is also an empty $6$-simplex of width 6 (in fact, it coincides with the smallest cyclotomic empty simplex of that dimension). 
That is:
\[
\A{4}{2} = \Cyc(4,101), \qquad
\A{6}{3} = \Cyc(6,6301).
\]
\end{remark}

An explicit formula for this circulant determinant is easy to compute.

\begin{lemma}
\label{lemma:volume}
For every even $d\ge 2$, and every $m\in \R$ we have
\begin{align*}
\Vol(\A{d}{m}) =\det(M(d,m)) =&  
\sum_{i=0}^{ d/2} \frac{d+1}{d+1-i}\binom{d+1-i}{i} m^{2i}\\
=& 1 + \sum_{i=1}^{ d/2} \frac{d+1}{i}\binom{d-i}{i-1} m^{2i}.
\end{align*}
\end{lemma}

\begin{proof}
The equality of volume and determinant follows from \cref{eq:volume}, since $\sum_i v_i =1$.
Let us compute the determinant via its permutation expansion.

The permutations that contribute to the determinant are \emph{exactly} (a) the cyclic shifts by one unit in both directions
(the diagonals of $-m$'s and $m$'s) which contribute $-m^{d+1}$ and $m^{d+1}$ respectively, so that they cancel out; and (b)
permutations consisting of a product of disjoint transpositions of cyclically consecutive elements. The corresponding terms consist in picking the SW-NE diagonal of certain disjoint $2\times2$ principal minors $\begin{pmatrix} 1 & -m \\ m & 1 \end{pmatrix}$, plus diagonal elements in the rest of positions. Each $2\times2$  minor contributes $m^2$ to the product, and there are exactly 
$\frac{d+1}{d+1-i}\binom{d+1-i}{i}$ ways of choosing $i$ disjoint pairs of  consecutive entries in a cyclic string of length $d+1$ (matchings of size $i$ in a cycle of length $d+1$; these are counted by Lucas polynomials, OEIS A034807). 

The proof that no other permutation contributes to the determinant follows from the following properties:
\blue{
\begin{itemize}
\item If a permutation that contributes uses two consecutive entries from the $-m$ diagonal  or  from the $m$ diagonal, then it uses the whole diagonal; that is, the permutation is the cyclic shift in one or the other direction.

\item Apart from the cyclic shifts, a permutation using a certain $-m$ must also use the  $m$ in the position symmetric to it, and vice-versa.\qedhere
\end{itemize}
}
\end{proof}

\begin{remark}
The same proof gives that the circulant determinant with generator
$v=(a,b,c,0,\dots,0)$ equals 
\[
a^{d+1} + c^{d+1} +(-1)^d \sum_{i=0}^{\lceil d/2\rceil} \frac{d+1}{d+1-i}\binom{d+1-i}{i} (-ac)^i b^{d+1-2i}.
\]
\end{remark}

\subsection{Width of \texorpdfstring{$\A{d}{m}$}{S(d,m)}}

\begin{theorem} 
\label{thm:circulant}
The lattice width of $\A{d}{m}$  equals $1$ if $d$ is odd and $2m$ if $d$ is even.
\end{theorem}

\begin{proof}
	\newcommand{\<}{\langle}
	\renewcommand{\>}{\rangle}

If $d$ is odd (that is, if the number of coordinates is even) then the sum of the even coordinates is a lattice functional giving width one to $\A{d}{m}$.

So, for the rest of the proof we assume $d$ even. 
The width of $\A{d}{m}$ is clearly $\le2m$, attained by any of the coordinate functionals. 
Let $w \in \ZZ^{d+1}$ be a non-constant vector, which we regard as a functional, and let us show that it gives width at least $2m$ to $\A{d}{m}$.

\blue{Recall that} $\vi$ denotes the $i$-th column of $M{d}{m}$. Indices $i$ go from $0$ to $d$ and are considered modulo $d+1$.
Observe that 
\[
\<w_,\vi\>=m(w_{i+1}-w_{i-1}) +w_i.
\]

		Since $d+1$ is odd, we can cyclically go through all the $i$'s in steps of length two \blue{(that is, skipping every second element)}. Thus, there exist indices $i_0$ and $j_0$ with $w_{2i_0}= \min (w)$ and
		\begin{align*}
		w_{2(i_0-1)} > w_{2i_0} = w_{2(i_0+1)} = \dots = w_{2j_0} < w_{2(j_0+1)}.
		\end{align*}
		(Observe that it could well be that $i_0=j_0$, but this is not a problem for us).
		We distinguish two cases:

		\begin{description}
			\item[Case 1: $w_{2i_0-1} \leq w_{2j_0+1}$]
				In this case, we compute
				\begin{align*}
				\<w, v^{(2j_0+1)}\> - \<w, v^{(2i_0-1)}\> 
					&= \quad m\underbrace{(w_{2j_0+2} - w_{2j_0})}_{\geq 1} \\
					&\quad-m\underbrace{(w_{2i_0} - w_{2i_0-2})}_{\leq -1} \\
					&\quad+ \underbrace{w_{2j_0+1} - w_{2i_0-1}}_{\geq 0} 
				\qquad\qquad \geq 2m. 
				\end{align*}
			\item[Case 2: $w_{2i_0-1} > w_{2j_0+1}$] 
				In this case, there exists a $k_0$ in $[i_0, j_0]$ (considered cyclically) such that 
				$w_{2k_0-1} > w_{2k_0+1}$ and (by construction) $w_{2k_0} = \min (w)$.
				We compute that
				\begin{align*}
				\<w, v^{(2j_0+1)}\> - \<w, v^{(2k_0)}\> 
					&=\quad m\underbrace{(w_{2j_0+2} - w_{2j_0})}_{\geq 1} \\
					&\quad- m\underbrace{(w_{2k_0+1} - w_{2k_0-1})}_{\leq -1} \\
					&\quad+ \underbrace{w_{2j_0+1} - w_{2k_0}}_{\geq 0 } 
				\qquad\qquad \geq 2m.
\qedhere
				\end{align*}
		\end{description}
\end{proof}

\begin{example}
\label{exm:nonprime_circulant}
The situation for odd $d$ generalizes as follows to any $d$ such that $d+1$ is not prime:
let $d+1=ab$ be a non-trivial factorization and consider the circulant $d$ simplex $S^{a}(d,m)$ generated by
\[
\vi[0] = (1,0,\dots,0,m,0,\dots,0,-m),
\]
where the first block of zeroes has length $a-2$. Observe that $S^{2}(d,m)=\A{d}{m}$.

\blue{Let $f_i: \R^{d+1} \to \R$, for $i\in \{0,\dots,a-1\}$, be the functional that adds the coordinates with indices equal to $i \mod a$.}
The simplex $S^{a}(d,m)$ has width one with respect to the functional $f_0$, and it is an empty simplex by the following argument. Let $F_a=(f_0,\dots,f_{a-1}):\R^{d+1} \to \R^a$.
 Let $F_b:\R^{d+1} \to \R^b$ be the projection that forgets all coordinates exept those with indices multiple of $a$. Then, $F_a(S^{a}(d,m))$ is the standard unimodular simplex, and the fiber containing $\vi[0]$ projects by $F_b$ isomorphically to the standard simplex. Since all vertex-fibers are isomorphic, $F_a$ sends $S^{a}(d,m)$ to a unimodular simplex and  \blue{all of its vertex fibers are unimodular}. This implies that $S^{a}(d,m)$ is empty.

Hence, if $d+1$ is not prime, then there are infinitely many empty circulant simplices of dimension $d$.
\end{example}

\subsection{The inequality description of \texorpdfstring{$\A{d}{m}$}{S(d,m)}}

Let $\uj[j] \in \Z^{d+1}$ be the $j$-th row of the matrix (where we use indices $j=0,\ldots,d$)
\[
\det(M(d,m)) \cdot M(d,m)^{-1}.
\] 
The factor $\det(M(d,m))$ makes any $\uj[j]$ an integer vector whenever $m$ is an integer. Clearly, the inequality description of $\A{d}{m}$ is
\[
\A{d}{m} = \left\{x \in \R^{d+1} \,:\, \sum_{i=0}^{d} x_i = 1,\, \text{ and } \langle \uj, x \rangle \ge 0 \ \forall j=0, \ldots, d\right\}.
\]

In order to explicitly compute the vectors $u\up j$, let us define the following \emph{continuant} matrices of size $k\times k$, for each $k\in \N$.
\[ C(k,m) := \begin{pmatrix}
	1      & -m     & 0      & \ldots & \ldots & 0      & 0     \\
	m      & 1      & -m     & \ddots &        &        & 0      \\
	0      & m      & 1      & \ddots & \ddots &        & \vdots \\
	\vdots & \ddots & \ddots & \ddots & \ddots & \ddots & \vdots \\
	\vdots &        &\ddots   & \ddots & 1      & -m     & 0      \\
	0      &        &        & \ddots & m      & 1      & -m     \\
	0     & 0      & \ldots & \ldots & 0      & m      & 1
\end{pmatrix}
\]
and the following \emph{continuant numbers}
\[c_k := c_k(m):=\det(C(k,m)).\]
We additionally define $c_{-1} := 0$, $c_{0} := 1$, and observe that $c_1 = 1$ and $c_2 = 1+m^2$. We  have:

\begin{lemma} Let $k \ge 1$. Then the following holds:
\begin{enumerate}
\item \[c_k = c_{k-1} + m^2 c_{k-2}.\]
\item \[c_k = \sum_{i=0}^{\lfloor\frac{k}{2}\rfloor} \binom{k-i}{i} m^{2i}=\sum_{i=0}^{\infty} \binom{k-i}{i} m^{2i}.\]
\end{enumerate}
\label{c-lemma}
\end{lemma}

\begin{proof}
 (1) easily follows from Laplace expansion. For (2) it is straightforward to verify the recursion in (1). Alternatively, (2) can be proved directly with the same arguments as in Lemma~\ref{lemma:volume}, the main difference being that now the $2\times 2$ minors are chosen from a linear string, rather than a circular one.
\end{proof}

It follows from the lemma that $c_k=c_k(m)$ is a polynomial in $m$ of degree $k$ or $k-1$ (depending on the parity of $k$)  with constant coefficient $1$. In fact, we later relate them to the well-studied Fibonacci polynomials.

Since $M(d,m)$ is circulant, its inverse is circulant too. Thus, each $\uj$ is indeed the $j$th cyclic shift of the first row $\uj[0]$ of $\det(M(d,m)) \cdot M(d,m)^{-1}$. Using the continuant numbers we can give an explicit formula for $\uj[0]$:

\blue{For the rest of this section, let $u =(u_0, \ldots, u_{d})$ denote the first row of the matrix $\det(M(d,m)) \cdot M(d,m)^{-1}$.
Analyzing the entries of $u$ is crucial in what follows.}

\begin{proposition}
\[
u_k = c_{d-k} \, m^{k} + (-1)^{d+k-1} c_{k-1}\, m^{d+1-k},
\qquad \forall\  k \in \{0, \ldots, d\}.
\]
In particular, if $d$ is even we get:
\begin{enumerate}
\item $u_k >0$ for every odd $k=1,3,\dots,d-1$.
\item $u_k > u_{k+2}$ for every even $k=0,2,\dots,d-2$.
\end{enumerate}
\label{u-prop}
\end{proposition}

\begin{proof}
For a matrix $M$, denote by $M_{\setminus(i,j)}$ the matrix $M$ with the $i$th row and the $j$th column removed (with indices in $\{0, \ldots, d\}$). For the first formula, observe that $u_k$ equals $(-1)^k$ times the determinant of $M(d,m)_{\setminus(k,0)}$. If $k=0$, we immediately see that $\det(M(d,m)_{\setminus(0,0)}) = c_d$. So, let $k  > 0$. Now, we use Laplacian expansion along the first row of $\det(M(d,m)_{\setminus(k,0)})$. We get two summands (corresponding to the first and the last entry of the first row). The first one equals $(-m)$ times the determinant of a block tridiagonal matrix that is easily seen to be equal to $(-m)^{k-1} \cdot c_{d-k}$. The second summand equals $(-1)^{d+1} m$ times the determinant of a block tridiagonal matrix that evaluates to $c_{k-1} \cdot m^{d-k}$. Together we get that $u_k$ equals 
\[
(-1)^k ((-m) \cdot  (-m)^{k-1} \cdot c_{d-k} + (-1)^{d+1} m \cdot c_{k-1} \cdot m^{d-k}),
\]
 as desired.

For the second part, positivity for odd $k$  is obvious, since in this case $u_k$ is a positive combination of two of the $c_i$'s. 
To deal with even $k$, observe that denoting $\tilde c_k = m^{d-k}c_k$ the recurrence  in Lemma~\ref{c-lemma} becomes
\[
\tilde c_k = \frac{\tilde c_{k-1}}m + \tilde c_{k-2}.
\]
Hence, both the even and the odd subsequences of $\tilde c_k$ are strictly increasing. 
This implies that $u_k$ decreases with $k$ for even $k$, since  for even $k$ we have
\[
u_k = 
c_{d-k} \, m^{k} - c_{k-1}\, m^{d+1-k}=
\tilde c_{d-k} -  \tilde c_{k-1}.
\qedhere
\]
\end{proof}


\subsection{Emptiness of \texorpdfstring{$\A{d}{m}$}{S(d,m)}}

In this section we treat $m\in [0,\infty)$ as a continuous parameter. For positive small $m$  the standard basis vectors lie outside $\A{d}{m}$; for example, the functional $w=(1, 0,\dots,0,\epsilon)$ separates $e_1$ from all the $\vi$, whenever $\epsilon < 1-m$. In the other extreme, for $m\to \infty$ the $e_i$'s lie inside $\A{d}{m}$. 
With the above computations it is now straightforward to see where the limit between these two situations is. The non-obvious and maybe surprising part of the following result is that containing the standard basis is the only obstruction for $\A{d}{m}$ to be empty.

\begin{theorem}
\label{thm:empty}
Assume $d$ even.
The following conditions are equivalent:
\begin{enumerate}
\item $m^{d-1} < c_{d-1}$.
\item $u_d <0$.
\item $\A{d}{m}$ does not contain the standard basis vectors.
\item $\A{d}{m}$ is an empty simplex.
\end{enumerate}
\end{theorem}

\begin{proof}
The implication (4)$\;\ra\;$(3) is trivial, and the equivalence of the two first conditions follows from $u_d = m^{d} - c_{d-1}\, m$.

By cyclic symmetry, one of the standard basis vectors lies in $\A{d}{m}$ if and only if all of them lie. This happens if and only if $\langle u, e_i\rangle = u_i\ge 0$ for every $i$. Since $e_d$ is the smallest of the even $u_k$s and all odd $u_k$s are positive, we have the equivalence of (2) and (3).

Only (2)$\;\ra\;$(4) remains to be proved. This implication requires some additional notation, so we write it as a separate statement below (Corollary~\ref{coro:empty}).
\end{proof}

To show that $\A{d}{m}$ is empty whenever $u_d<0$ we introduce the following notation: 

Let $e_0, \ldots, e_d$ denote the standard basis of $\R^{d+1}$. We will again use indices modulo $d+1$. For $i=0, \ldots, d$, let $a_i = e_{i+1} - e_{i-1}$, so that the vertices of $\A{d}{m}$ are the points $e_i + m a_i$. 
Let $A$ be the convex hull of all the $a_i$'s, which is a cyclically symmetric $d$-simplex in the hyperplane $\{\sum_i x_i =0\} \subset \R^{d+1}$ that contains the origin in its relative interior. Let $F_i$ be the facet of $A$ opposite to the vertex $a_i$ and let $A_i$ be the convex hull of the origin and $F_i$. Observe that the simplex $A$ is barycentrically subdivided into the $d+1$ simplices $A_i$.
Finally, let $\Delta= \conv (e_0,\dots, e_d)$ denote the standard unimodular simplex in the hyperplane $\{\sum_{i=0}^d x_i =1\}$.

\begin{lemma}
\label{lemma:empty}
With this notation, and assuming $d$ even:
\begin{enumerate}
\item $S(d,m)$ is contained in the Minkowski sum $\Delta + m A$. 

\item Every lattice point of $\Delta + m A$ lies in  $e_{i-1} + m A_{i}$ for some $i$ (with indices taken modulo $d+1$). 

\item If $u_d < 0$ then $\langle u, v\rangle \le 0$ for every vertex $v$ of $e_d + m A_0$. Equality holds only for $v= e_d + m a_d$.

\end{enumerate}
\end{lemma}

\begin{figure}[htb]
\raisebox{0.1\height}{
\includegraphics[scale=0.65]{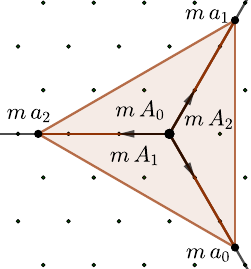}\quad}
\includegraphics[scale=0.65]{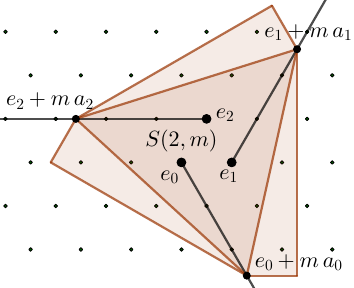}\quad
\includegraphics[scale=0.65]{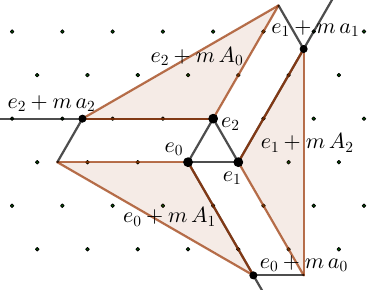}

\quad
$ m\,A = \bigcup_{i=0}^{d} m\,A_i$,\qquad 
$\A{d}{m} \subset \Delta + m\,A$, 
\qquad\qquad\qquad
$\Delta + m\,A$
\qquad\qquad
\qquad

\caption{Illustration of parts (1) and (2) of Lemma~\ref{lemma:empty}.
Left: the simplex $m\,A$ and its decomposition into $d+1$ dilated unimodular simplices $m\,A_i$. Center: our circulant simplex $\A{d}{m}$,  contained in the Minkowski sum $\Delta + m\,A$. Right: the simplices $e_{i-1} + m A_{i}$ contain all the lattice points in $\Delta + m\,A$. In the picture we have $m=2.6$, which is bigger than the emptyness threshold in dimension two, $m_0(2)=1$.}
\end{figure}

\begin{figure}[htb]
{\includegraphics[scale=1.5]{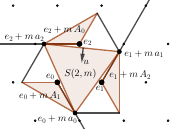}\quad}

\caption{Illustration of part (3) of Lemma~\ref{lemma:empty}. In this picture $m=0.8$, below the threshold 
$m_0(2)=1$. The simplex $e_{d} + m A_{0}$ lies completely in the open half-space $\{u <0\}$, except for its  vertex $e_{d} + m a_{d}$, common to $\A{d}{m}$. }
\end{figure}

\begin{proof}
Part (1) is obvious, since the vertices of $S(d,m)$ are the points $e_i +m  a_i$.

Part (3) uses the description of the vector $u$.
The vertices of $e_d + m A_0$ are the points $e_d +m  a_i$ for $i\ne 0$, together with $e_d$. We have $\langle u, e_d \rangle = u_d < 0$, and for the \blue{other vertices of $e_d + m A_0$} we write
\[
\langle u, e_d +m  a_i\rangle = 
\langle u, e_d -e_i\rangle + 
\langle u, e_i +m  a_i\rangle. 
\]
The second summand equals zero since $e_i +m  a_i$ (for $i\ne 0$) is a vertex of the facet of $\A{d}{m}$ given by $\langle u , x \rangle=0$. The first summand equals $u_d-u_i$, which is negative except for $i=d$, by parts (1) and (2) in Proposition~\ref{u-prop} (assuming $d$ even).

It only remains, then, to prove part (2).
Let $C=\frac1{d+1}(1,\dots,1)$ be the ``center'' of the hyperplane $\{\sum_i x_i =1\}$; that is, the barycenter of both $\A{d}{m}$ and $\Delta$. 

Let $\Sigma_i = \cone(A_i)$ be the $i$th facet-cone of $A$, and consider the complete fan centered at $C$ given by the affine cones $C+\Sigma_i $, $i=0,\dots,d$. 
We will represent points in the \(\sum_{i=0}^d x_i=1\) hyperplane in the form 
\begin{equation}
e_d + \sum_{i=1}^d \lambda_i a_i\label{eq:hyper}.
\end{equation}
Part (2) is equivalent to the claim that all the lattice points in $(C+\Sigma_i ) \cap (\Delta + m A)$ lie in $e_{i-1} + m A_i$. 
\blue{We are going to} prove the claim for \(i=0\); that is, we show that the lattice points in $(C+\Sigma_0 ) \cap (\Delta + m A)$ 
\blue{are contained in}  $e_{d} + m A_0$. The claim for all \(i\) follows cyclically.

The cone $\Sigma_0$, generated by $a_1,\dots, a_{d}$, is unimodular. Thus, its lattice points are the semigroup $Z$ generated by $a_1,\dots, a_{d}$: 
\[
Z:=\left\{\sum_{i=1}^d \lambda_i a_i: \lambda_i \in \Z_{\ge0} \ \forall i\right\}.
\] 
This implies that the lattice points in the \blue{translated cone $C+\Sigma_0 $} are the set $p+Z$ where $p$ is the unique lattice point in the \blue{translated} half-open parallelepiped 
\[
C +\left\{\sum_{i=1}^d \lambda_i a_i: \lambda_i \in [0,1) \ \forall i\right\}.
\]
This unique point $p$ is $e_d$, since
\[
e_d = C + \sum_{i=1}^d \frac{i}{d+1} a_{2i}.
\]
(Remember that indices are regarded modulo $d+1$, so $a_{2i}$ for $i > d/2$ means $a_{2i-d-1}$).
\blue{Summing up, we have so far proved that
\begin{equation}
\Z^d \cap (C+\Sigma_0) = e_d + Z.
\label{eq:CandZ}
\end{equation}


Recall that we denote $F_0$ the facet of $A$ opposite to $a_0$, which equals the facet of $A_0$ opposite to $C$.
Let $f_0$ be its outer normal vector and let $b_0$ be the value of $f_0(x)$ on $F_0$. We have that $f_0( a_j)=b_0$ for all $j\ne 0$ and $f_0( a_0) < b_0$. 

The facet $F'_0$ of \(\Delta +m A\) with this same normal vector $f_0$ equals \(e_i + m F_0\), where \(e_i\) is the vertex of $\Delta$ maximizing $f_0 (e_i)$.
Since
\[
f_0(a_j) = f_0(e_{j+1}-e_{j-1})= f_0(e_{j+1}) - f_0(e_{j-1}) >0 \quad  \forall j \ne 0,
\]
the $e_i$ maximizing $f_0$ can only be $e_{-1} =e_d$.
Thus, \(F_0'\) equals \(e_d+ m F_0\) and its facet inequality is $f_0(x) \le f_0(e_d) + m b_0$. In particular,
\begin{equation}
\Delta +m A \subset\{f_0(x) \le f_0(e_d) + m b_0\}.
\label{eq:f_0}
\end{equation}

Let now $p$ be a lattice point in $(C + \Sigma_0) \cap (\Delta +m A)$. By \eqref{eq:CandZ} we have that
$p=e_d +\sum_{i=1}^d\lambda_i a_i$ for some $\lambda_1,\dots, \lambda_d\in \Z_\ge 0$. This leads to
\[
f_0 (p) = f_0(e_d) + b \sum_{i=1}^d\lambda_i
\]
and \eqref{eq:f_0} implies $\sum_{i=1}^d\lambda_i\le m$. Thus, $p\in e_d +mA_0$,
as we wanted to show.
%
%
}
\end{proof}

From this lemma, we quickly obtain that \((2) \implies (4)\) in Theorem~\ref{thm:empty}.
Indeed, by part (1) of the lemma all lattice points of \(\A{d}{m}\) lie in $\Delta + mA$ and by part (2) they also lie in $e_{i-1} + m A_{i}$ for some $i$. Part (3) shows that the only point in both  $e_{d} + m A_{0}$ and $\A{d}{m}$ is the vertex $e_d+ma_ d$. Hence, by cyclic symmetry the only lattice points that \(\A{d}{m}\) contains are its vertices, the definition of empty.

Observe that the equation $m^{d-1} = c_{d-1}$ in part (1) of  Theorem~\ref{thm:empty} has a unique solution with $m>0$, since the right-hand side is a polynomial of degree $d-2$ with all coefficients nonnegative (and not all zero). Thus, the following is a more explicit version of Theorem~\ref{thm:empty}:

\begin{corollary}
\label{coro:empty}
Let $d$ be even and let $m_0(d)$ be the unique positive solution of $u_{d}(m)=0$. Equivalently, of
\[
m^{d-1} = c_{d-1}.
\]
If $m <m_0(d)$ then $\A{d}{m}$ is empty.
\qed
\end{corollary}

We have computationally verified emptyness up to $d=22$.

\subsection{Asymptotics of \texorpdfstring{$m_0(d)$}{mâ(d)}, via Fibonacci polynomials and hyperbolic functions}
\label{sec:asymptotics}
We now investigate the asymptotics of \(m_0(d)\), to understand how wide an empty circulant simplex \(\A{d}{m}\) may be. We continue to use the base assumption that \(d\) is even, since otherwise the width is \(1\).

The continuant numbers and thus the coefficients of the $u$-vector have many interesting interpretations in terms of special functions (e.g, generalized Binomial polynomials, hypergeometric functions, etc. \cite[p.203,204,213]{ConcreteMath94}). Here, we  stress the relation to the following sequence of polynomials, one for each degree $h$. Following \cite{HogBic73,HogLon74} we call them \emph{Fibonacci polynomials}:%
\footnote{
Other sources (e.g., the book \cite{FlajSedge}) call Fibonacci polynomials a related, but different, sequence.}
%
%
\[
\Fibo_h(x):=\sum_{j=0}^{\infty} \binom{h-1-j}{j} x^{h-1-2j}.
\]
As usual, in the definition we take the natural convention that $\binom{a}{b}$ is well-defined for every $a\in \Z$ and $b\in \N$, with $\binom{a}{b}=0$ if $a <b$. From Lemma~\ref{c-lemma} we see that
 \[
 c_{k} = \sum_{i=0}^{\infty} \binom{k-i}{i} m^{2i}=m^{k}\Fibo_{k+1}\left(\frac1{m}\right)
 \]
and from Proposition~\ref{u-prop} we derive that for $k \in \{0, \ldots, d\}$ 
\begin{equation}
\label{eq:Fibo_u}
u_{k} = \left(\Fibo_{d+1-k}\left(\frac{1}{m}\right)+(-1)^{k-1} \Fibo_{k}\left(\frac{1}{m}\right)\right) \cdot m^d.
\end{equation}

The property of $\Fibo_h$ that we need is the following~\cite{HogBic73}:
\begin{equation}
\label{eq:Fibo}
\Fibo_{2n}(2\sinh z) = \frac{\sinh (2nz)}{\cosh(z)}, \quad
\Fibo_{2n+1}(2\sinh z) = \frac{\cosh ((2n+1)z)}{\cosh(z)}.
\end{equation}

Let us define
\[\alpha = \alpha(m) := \arcsinh\left(\frac{1}{2m}\right),\]
so that $\frac1m=2\sinh \alpha$.
Note that $\cosh(\alpha)=\cosh(\arcsinh\left(\frac{1}{2m}\right)) = \sqrt{\frac{1}{4m^2}+1}$.

Using Equations~\eqref{eq:Fibo_u} and~\eqref{eq:Fibo} we get that, for  even \(d\),
\[
  \frac{u_d}{m^d} =  F_1\left(\frac1m\right) - F_d\left(\frac1m\right) =1 - \frac{\sinh (d \alpha)}{\cosh(\alpha)}.
\]

Since $m_0(d)$ is the unique solution to $u_d=0$, we obtain that $\alpha_0(d):=\alpha(m_0(d))$ is the unique solution to 
$\cosh\left(\alpha\right) - \sinh\left(d \alpha\right) =0$.

\begin{theorem}
\label{thm:asymptotics}
\[
\lim_{d\to \infty} \frac{2m_0(d)}{d} = \frac{1}{\arcsinh(1)} \simeq 1.1346.
\]
\end{theorem}

\begin{proof}
The equation
\[
\sinh (d\alpha_0)= \cosh (\alpha_0)
\]
clearly implies $\lim_{d\to \infty} \alpha_0(d)=0$. Thus 
we can  approximate $\sinh(\alpha_0)\sim \alpha_0$ and $\cosh(\alpha_0)\sim 1$
in the standard sense that $a\sim b$ means $\lim_{d\to \infty} a/b =1$.
Hence:
\[
\sinh(d\alpha_0) \sim 1 
\ \ \Rightarrow \ \ 
d\alpha_0 \sim \arcsinh(1)
\ \ \Rightarrow \ \ 
\frac1{2m_0(d)} = \sinh(\alpha_0) \sim  \alpha_0 \sim \frac{\arcsinh(1)}d.
\]
\end{proof}

\begin{remark}
This proof also implies that $2m_0(d)$ is always \emph{smaller} than the asymptotic value $\frac{d}{\arcsinh(1)} \simeq 1.1346 d$, because all the ``$\sim$'' in the proof are in fact ``$>$''. 
But (experimentally) the convergence is fairly quick. 
Up to at least $d=1000$ we have
\[
\lfloor m_0(d) \rfloor = \left\lfloor\frac{d}{2\arcsinh(1)}\right\rfloor.
\]
\end{remark}

\subsection{The boundary of \texorpdfstring{$\A{d}{m}$}{S(d,m)}}

We here again assume $m$ to be an integer, so that $\A{d}{m}$ is a lattice simplex.
It is easy to show that it is not always cyclic:

\begin{proposition}
If $d+1\equiv0\pmod 3$ and $m$ is odd, then all facets of $\A{d}{m}$ have even volume. Hence $\A{d}{m}$ is not cyclic.
\end{proposition}

\begin{proof}
Assume $d+1\equiv 0 \pmod 3$ and let $p$ be (the class modulo $\Lambda_{\A{d}{m}}$ of) a point with barycentric coordinates $\frac12(1,1,0,1,1,0,1,1,0,\dots)\pmod1$. Its cartesian coordinates in $\R^{d+1}/\Z^{d+1}$ are $\frac12(1-m,1+m,0,1-m,1+m,0,1-m,1+m,0, \dots)$, which is an integer vector if (and only if) $m$ is odd. By construction, $p + \Lambda_{\A{d}{m}}$ intersects the affine span of every third facet. This implies those facets (and, by cyclic symmetry, all facets) have even volume.

A simplex in which all facet volumes have a common factor cannot be cyclic, by Proposition~\ref{prop:cyclic}.
\end{proof}

In what follows we show that the values of $d$ and $m$ in this proposition are the only ones for which $\A{d}{m}$ is \emph{not} cyclic (assuming $d$ even) and that when this happens the quotient group $G_{\A{d}{m}}$ is isomorphic to $\Z_{N/2} \oplus \Z_2$. The key fact is that the volume of each facet of $\A{d}{m}$, which does not depend on the facet, by cyclic symmetry, equals the $\gcd$ of the vector $u$ from the previous sections. (See the proof of Theorem~\ref{thm:boundary} for an explanation).
Our next result computes this $\gcd$:

\begin{lemma}
For every $d\in 2\N$ and  $m\in \Z$ the following conditions are equivalent:
\begin{enumerate}
\item $\gcd(u_0,\ldots,u_{d}) = 2$.
\item $\gcd(u_0,\ldots,u_{d}) \not=1$.
\item $\gcd(u_0,u_2,u_{d}) \not= 1$.
\item $d \equiv 2 \pmod 6$ and $m$ is odd.
\end{enumerate}
\label{gcd-lemma}
\end{lemma}

%
 
\begin{proof}

(1) $\ra$ (2) $\ra$ (3) are clear.
In the following we will often use Lemma~\ref{c-lemma} and Proposition~\ref{u-prop} without further mentioning.

To show (3) $\ra$ (4), if $p$ be a prime that divides $u_0,u_2,u_{d}$ then $u_0 = c_d=c_{d-1}+m^2 c_{d-2} \equiv 0 \pmod p$. As $c_d \equiv 1 \pmod m$,  $p$ does not divide $m$, and 
\begin{equation}
c_{d-1} \equiv -m^2c_{d-2} \pmod p
\label{peq1}
\end{equation}
Next, $u_2 = c_{d-2}m^2-m^{d-1} \equiv 0 \pmod p$, hence, as $p \not|m$, 
\begin{equation}
c_{d-2} \equiv m^{d-3} \pmod p
\label{peq2}
\end{equation}
Finally, $u_{d}=m^d - c_ {d-1}m \equiv 0 \pmod p$, hence, again as $p \not|m$, 
\begin{equation}
c_{d-1} \equiv m^{d-1} \pmod p
\label{peq3}
\end{equation}
From Equations~\eqref{peq1},\eqref{peq2},\eqref{peq3} we get $m^{d-1} \equiv -m^{d-1} \pmod p$. As $p$ does not divide $m$, we see $p=2$. Now, $c_d \equiv 1 \pmod m$ implies $m$ is odd and $c_d=u_0$ is even. This gives $d\equiv 2\pmod 3$ since, for odd $m$, Lemma~\ref{c-lemma}(1) gives the Fibonacci recursion
\[
c_k \equiv c_{k-1} + c_{k-2} \pmod 4,  \qquad c_{-1}=0, \quad c_0=1,
\]
which implies
\begin{equation}
c_k \text{ is even } \;\lra\; k \equiv 2 \pmod 3.
\label{modness} 
\end{equation}
Since $d$ is even, $d \equiv 2 \pmod 3$ implies $d \equiv 2 \pmod 6$, as desired.

It remains to prove (4) $\ra$ (1). 
Let $d \equiv 2\pmod 6$ and $m$ be odd. Let $k \in \{0, \ldots, d\}$. From Proposition~\ref{u-prop} we have $u_{k} = c_{d-k} \, m^{k} + (-1)^{k-1} c_{k-1}\, m^{d+1-k} \equiv c_{d-k} + c_{k-1} \pmod 2$, as $m$ is odd. Therefore, from the equivalence \eqref{modness} it is straightforward to verify that $d \equiv 2 \pmod 6$ implies that $u_{k}$ is even for every $k$. So, $2$ divides $\gcd(u_0,\ldots,u_{d})$. From the argument in the proof of (3) $\ra$ (4), we know that $2$ is the only prime divisor of $\gcd(u_0,\ldots,u_{d})$. Hence, to show (1) it suffices to show that $4$ does not divide $\gcd(u_0,\ldots,u_{d})$. 
This follows again from the Fibonacci recursion modulo $4$, which implies 
\[c_k \equiv 0 \pmod 4 \;\lra\; k \equiv 5 \pmod 6.\]
Thus, if  $u_0 = c_d \equiv 0 \pmod 4$ then $d \equiv 5 \pmod 6$, a contradiction.
\end{proof}

Still, even in the non-cyclic case the facets of $\A{d}{m}$ (for $d > 2$) are empty, no matter how large $m$ is. 

\begin{theorem} Let $d$ be even and $m$ be an integer. 
Then
\begin{enumerate}
\item If $d=2\pmod 6$ and $m$ is odd then all facets of $\A{d}{m}$ have volume two and the quotient group $G_{\A{d}{m}}$
is isomorphic to $\Z_{N/2}\oplus \Z_2$.
The facets are still empty (unless $d=2$).

\item Otherwise, facets of $\A{d}{m}$ are unimodular. Hence, $\A{d}{m}$ is cyclic.
\end{enumerate}
\label{thm:boundary}
\end{theorem}

\begin{proof}
Observe that $u$ is, modulo signs, the vector of $d \times d$-minors of the matrix obtained by deleting the first column from $M(d,m)$. Thus, $\gcd(u_0, \ldots, u_d)$ equals the volume of the facet $F$ opposite to the first vertex of $\A{d}{m}$. Since, by cyclic symmetry, all facets have the same volume,  Lemma~\ref{gcd-lemma} implies most of the statement:  if $d \ne 2 \pmod 6$ or $m$ is even, then all facets are unimodular, hence $\A{d}{m}$ is cyclic. 
If $d = 2 \pmod 6$ and $m$ is odd, then all facets have volume two, so the exact sequence of Proposition~\ref{prop:notcyclic} becomes
\[
0 \to \Z_2 \to G_{\A{d}{m}} \to \Z_{N/2} \to 0.
\]
Since $G_{\A{d}{m}}$ is not cyclic (by Proposition~\ref{prop:cyclic}) this implies $G_{\A{d}{m}}\cong \Z_{N/2}\oplus \Z_2$.

To show that facets are still empty, observe that $\Vol(F)=2$ implies that every barycentric coordinate of every lattice point in $\aff(F)$ is in $\frac12\Z$. Hence, if $F$ was not empty then it would contain a lattice point with two barycentric coordinates equal to $\frac12$ and the rest equal to $0$. That is to say, the mid-point of an edge would be a lattice point, which clearly is not the case in $\A{d}{m}$, when $d\ge 3$. 
\end{proof}

\medskip

\end{document}